\documentclass[12pt, a4paper]{article}
\usepackage{amsmath, amssymb, amsthm}
\usepackage{graphicx}
\usepackage{mathrsfs} 
\usepackage[mathcal]{euscript}
\usepackage{yfonts}

\usepackage[left=1.2in,right=1.2in,top=1.4in,bottom=1.4in]{geometry}

\usepackage[colorlinks=true, citecolor=blue, linkcolor=red]{hyperref}

\numberwithin{equation}{section}

\theoremstyle{plain}
\newtheorem{theorem}{Theorem}
\newtheorem{lemma}{Lemma}
\newtheorem{proposition}{Proposition}
\newtheorem{corollary}{Corollary}

\theoremstyle{definition}
\newtheorem{remark}{Remark}
\newtheorem{example}{Example}

\title{On time-dependent boundary crossing probabilities of diffusion processes as differentiable functionals of the boundary\footnote{Supported
    by the University of Melbourne Faculty of Science Research Grant Support Scheme.}}

\date{}

\author{Vincent Liang\footnote{School of Mathematics and Statistics, The University of Melbourne, Parkville 3010, Australia; e-mail: liangvw@gmail.com (the corresponding author).}\,\, and Konstantin Borovkov\footnote{School of Mathematics and Statistics, The University of Melbourne, Parkville 3010, Australia; e-mail: borovkov@unimelb.edu.au.}}

\newcommand{\T}{T}

\DeclareMathOperator{\var}{Var}
\newcommand{\bm}[1]{\boldsymbol{#1}}

\newcommand{\contSpace}{C}

\newcommand{\boundarySpace}{C^2}

\newcommand{\cmSpace}{H}
\newcommand{\LpSpace}{L}

\newcommand{\Px}{\mathbf{P}}
\newcommand{\Ex}{\mathbf{E}}

\newcommand{\Pz}{\widetilde{\Px}}
\newcommand{\Ez}{\widetilde{\Ex}}
\newcommand{\Wz}{\widetilde{W}}

\newcommand{\heatX}{\mathscr{L}}
\newcommand{\heatY}{\overline{\mathscr{L}}}

\newcommand{\A}{\mathscr{A}} 

\newcommand{\X}{X} 
\newcommand{\Y}{Y} 
\newcommand{\Q}{Q_{\T}}
\newcommand{\D}{D_{\T}}

\newcommand{\boundarySet}{S}

\DeclareMathOperator{\cl}{cl}

\newcommand{\V}{v}
\newcommand{\bv}{\overline{v}}
\newcommand{\qY}{\overline{q}}
\newcommand{\q}{q}

\newcommand{\muX}{\mu}
\newcommand{\sigmaX}{\sigma}
\newcommand{\muY}{\overline{\mu}}
\newcommand{\sigmaY}{\overline{\sigma}}
\newcommand{\usigma}{\sigma_0}
\newcommand{\muZ}{\eta}
\newcommand{\tauX}{\tau}
\newcommand{\tauY}{\overline{\tau}}
\newcommand{\Eta}{\textnormal{H}}

\newcommand{\hd}{\mathfrak{h}}
\newcommand{\hD}{N}

\newcommand{\Lmg}{L}
\newcommand{\rieszMG}{U}

\newcommand{\zhe}[7]{#1: #2. {\textit{#3}}  {\textbf{#4}}, (#5), #6. MR#7}

\newcommand{\kn}[7]{#1: #2. {\textit{#3}}, #4, #5.  #6~pp. MR#7}

\newcommand{\arx}[3]{#1: #2,  #3 }

\begin{document}

\maketitle

\begin{abstract}
The paper analyses the sensitivity of the finite time horizon boundary non-crossing probability $F(g)$ of a general time-inhomogeneous diffusion process to perturbations of the boundary $g$. We prove that, for boundaries $g\in C^2,$ this probability is G\^ateaux differentiable in directions $h \in \cmSpace \cup \contSpace^2$ and Fr\'echet-differentiable in directions $h \in H,$ where~$H$ is the Cameron--Martin space, and derive a compact representation for the derivative of~$F$. Our results allow one to approximate~$F(g)$ using boundaries $\bar{g}$ that are close to $g$ and for which the computation of~$F(\bar{g})$ is feasible. We also obtain auxiliary results of independent interest in both probability theory and PDE theory. These include: (i)~an elegant probabilistic representation for the limit of the derivative with respect to $x$ of the boundary crossing probability when the process starts at point $(t,x)$ in the time-space domain and $x\uparrow g(t),$ and~(ii)
a Shiryaev--Yor type martingale representation for the indicator of the boundary non-crossing event for time-dependent boundaries.


\end{abstract}

\section{Introduction}

The probability $F(g)$ that a diffusion process $X$ stays beneath a time-dependent boundary $g$ during a given time interval $[0,T]$ is a classical object of study in probability theory. Computation of this quantity is a crucial step in solving several important applied problems in areas such as optimal stopping~\cite{Peskir2006}, genetic linkage analysis~\cite{Dupuis2000}, pricing barrier options~\cite{Armstrong2001,Roberts1997}, etc. In the general case, the values $F(g)$ can be approximated by using boundaries $\bar{g}$ that are close to $g$ and for which computation of $F(\bar{g})$ is feasible. A key element in this is analysis of the sensitivity of $F(g)$ to small perturbations of the boundary~$g$.

It has been known since the 1930s that the non-crossing probability can be expressed as a solution to the respective boundary value problem for the backward Kolmogorov partial differential equation \cite{Kolmogorov1931,Khintchine1933,Kolmogorov1933}. However, simple closed-form expressions for the solutions are mostly confined to the case of the Brownian motion process using the method of images~\cite{Lerche1986}, whereas available results for diffusion processes usually rely on verifying Cherkasov's conditions \cite{Cherkasov1957,Ricciardi1976,Ricciardi1983} and then transforming the problem to that for the Brownian motion process by using a monotone transformation. Outside this class of special cases, one should mostly rely on computing numerical approximations to the desired probabilities.

One popular approach consists in first computing the first-passage-time density for the boundary which can be done using Volterra integral equations. Much work was done on the method of integral equations for approximating the first-passage-time densities for general diffusion processes \cite{Buonocore1990,Buonocore1987,Giorno1989,Jaimez1991,Ricciardi1984,Ricciardi1983,Sacerdote1996}, culminating with~\cite{Gutierrez1997}, which expressed the first-passage-time density of a general time-inhomogeneous diffusion process in terms of the solution to a Volterra integral equation of the second kind. Volterra integral equations of the first kind for the first-passage time density for the Brownian motion were derived in \cite{Durbin1971,Loader1987,Park1976,Peskir2002}. Although the method of integral equations is quite efficient for computational purposes in the case of the Brownian motion process, a drawback of the method is that the kernel of the integral equation is expressed in terms of the transition probabilities of the process. Therefore, we have to first compute these transition probabilities for all times on the time grid used, which will be problematic for a general diffusion process. For further details on the connection between Volterra integral equations and first-passage-time densities, we refer the reader to~\cite{Peskir2002}. Other popular computational techniques are the Monte Carlo method \cite{Gobet2000,Ichiba2011}, including exact simulation methods \cite{Beskos2006,Beskos2005,Herrmann2019,Herrmann2020}, and numerical solving of partial differential equations (see e.g.~\cite{Patie2008} and references therein).

The classical probabilistic approach to the boundary-crossing problem is the method of
weak approximation, which involves proving that a sequence of simpler processes~$X_n$ weakly converges to the desired diffusion process~$X$ in a suitable function space, which entails the convergence of the corresponding non-crossing probabilities. Along with the already mentioned references \cite{Khintchine1933,Kolmogorov1931,Kolmogorov1933}, this approach was effectively used in~\cite{Erdos1946} in the case of `flat boundaries' (see also Chapter 2, \S 11 in \cite{Billingsley1968}). More recently, the authors in~\cite{Fu2010} approximated the Brownian motion process with a sequence of discrete Markov chains with absorbing states at the boundary and expressed the non-crossing probability in terms of the product of the respective transition matrices. The authors of~\cite{Ji2015} extended the results from~\cite{Fu2010} by approximating a general diffusion process with a sequence of Markov chains and similarly expressed the non-crossing probability via a product of the transition matrices. However, the convergence rates of these approximations were shown to be $O(n^{-1/2}),$ where the discrete time steps are of size~$n^{-1},$ which leaves much to be desired in practical applications. A Brownian bridge corrected Markov chain approximation for general diffusion processes was proposed in~\cite{Liang2023}, which was numerically demonstrated to converge at the rate $O(n^{-2}).$ One can also use continuity correction employing boundary shifting to compensate for the `discrete monitoring bias' \cite{Broadie1997,Gobet2010,Siegmund1982}.

Another standard approach is to approximate the original boundary~$g$ with a close one for which computing the non-crossing probability is more feasible, a popular choice being a piecewise linear approximation of~$g$. These approximations were used in~\cite{Durbin1971} in combination with the integral equations approach to obtain an approximation to the first passage time density in the case of the Brownian motion process. For this process, piecewise linear approximations for~$g$  were also used in~\cite{Wang1997} together with the explicit formulae for linear boundary crossing probabilities for the Brownian bridge to express the approximating probability as a multivariate Gaussian integral. This approach was extended in~\cite{Wang2007} to the case of diffusion processes that can be obtained as monotone transforms of the Brownian motion.

To meaningfully use approximations of this kind, one must analyse the sensitivity of $F(g)$ to small perturbations of the boundary~$g.$ In~\cite{Wang1997}, the authors proposed a special construction for a sequence of non-uniform partitions of $[0,T]$ for the nodes of piecewise linear approximations~$g_n^*$ of~$g$ that satisfied $\| g_n^* -g \|_{\infty} \to 0$ as $n\to\infty,$ where $\|x \|_{\infty} := \sup_{t\in[0,T]}\lvert x(t)\rvert,$ and obtained the following bound:
\begin{equation*}
	\lvert F(g_n^*) - F(g) \rvert  = O\bigl( \| g_n^* - g\|_{\infty}\sqrt{-\ln(\| g_n^* - g\|_{\infty})} \bigr).
\end{equation*}
Under the assumptions that $g''(t)$
exists and is continuous in $(0,T),$ $g''(0+)\neq 0,$ and $g''(t) =0$ at only finitely many points $t\in(0,T),$ the authors in~\cite{Potzelberger2001} showed that, for the same sequence of boundaries~$g_n^*,$ one actually has
\begin{equation*}
	\lvert F(g_n^*) - F(g)\rvert = O(n^{-2}), \quad n \to \infty.
\end{equation*}
Unfortunately, the special construction of the approximating functions~$g_n^*$ achieving the above convergence rate is quite cumbersome.

To remove this obstacle and obtain convergence rates for general approximations to~$g,$ the authors in~\cite{Novikov1999} obtained the following bound in the case when~$g$ is bounded measurable and~$h$ is differentiable with $h(0) = h(T) =0$:
\begin{equation*}
 \lvert F(g + h) - F(g) \rvert \leq \mbox{$\frac{1}{\sqrt{2\pi}}\Phi(g(T)/\sqrt{T}) (\int_0^T \dot{h}(t)^2\,dt)^{1/2}$},
\end{equation*}
where $\Phi$ is the standard normal distribution function. As a consequence, they concluded that, for piecewise linear~$g_n$ with equally spaced nodes $(k/n,g(k/n))$, $k=1,2,\ldots,n$, one has
\begin{equation*}
	\lvert F(g_n) - F(g)\rvert =  O(n^{-3/2}\sqrt{\ln{n}}),\quad n\to\infty.
\end{equation*}

The authors in~\cite{Borovkov2005} went further and proved that~$F$ is Lipschitz with respect to the uniform norm provided that  $\sup_{t,t+\delta \in[0,T]}\lvert g(t+\delta) - g(t)\rvert \leq K\lvert \delta\rvert$ for some $K \geq 0.$ They showed that, in this case,
\begin{equation*}
\lvert F(g+h) - F(g) \rvert \leq (2.5K + 2T^{-1/2}) \| h\|_{\infty}
\end{equation*}
for any bounded measurable $h,$ and concluded that, for $g \in \contSpace^2([0,T])$ (i.e.~$g$ is twice continuously differentiable on $[0,T]$) and the same piecewise linear~$g_n$ as above,
\begin{equation*}
	\lvert F(g_n) - F(g)\rvert \leq (0.313K + 0.25T^{-1/2})\| \ddot{g}\|_{\infty}n^{-2}.
\end{equation*}
This result was extended in~\cite{Downes2008} to time-homogeneous diffusion processes with a unit diffusion coefficient. An extension to the multivariate Brownian motion was established in~\cite{McKinlay2015}.

The above bounds indicated that $F(g)$ might be differentiable in some sense. Indeed, it was shown already in 1951 in~\cite{Cameron1951} that~$F$ is G\^ateaux differentiable in `directions' from the Cameron--Martin space $\cmSpace$ in the case when $X=W.$ In the general case of time-inhomogeneous multivariate diffusion processes, the G\^ateaux differentiability of a Feynman--Kac type expectation with respect to a boundary perturbation was established in~\cite{Costantini2006} together with an expression for the derivative. However, due to the continuity assumptions made in the paper, the result from~\cite{Costantini2006} does not apply to boundary-crossing probabilities. The authors in~\cite{Funaki2006} obtained an infinite-dimensional integration by parts formula for the Brownian measure restricted to the set of paths between two curves, which can be used to obtain the G\^ateaux derivative of~$F$ in directions $h \in \cmSpace$ in the case of the Brownian motion. For time-homogeneous diffusion processes, the authors in~\cite{Borovkov2010} proved that~$F$ is G\^ateaux differentiable in directions $h\in \contSpace^{2}([0,T]),$ possibly with $h(0) \neq 0$ (unlike the case when $h \in \cmSpace$), and obtained a closed-form expression for the G\^ateaux derivative in terms of an expectation of a Brownian meander functional.

The objective of the present paper is to establish the existence of and obtain a representation for the Fr\'echet derivative of the time-dependent boundary non-crossing probability of a general time-inhomogeneous diffusion process. We show that in the case of time-homogeneous diffusions, our representation is equivalent to the one obtained in~\cite{Borovkov2010} which used a different technique in that special situation. However, even in that special case, our new representation is much more convenient for computational purposes. Moreover, our extension of the class of admissible directions $h$ to the Cameron--Martin space makes it possible, when approximating $g$ with a piecewise linear $\bar{g},$ not only bound the approximation error as in \cite{Downes2008} but also get a higher-order approximation using the derivative of $F$ in the direction $\bar{g} - g.$

In the course of the proof, we obtained two identities that are of independent interest. Namely, Lemma~\ref{prop:VL_BD} provides a probabilistic expression for the limit of the spatial derivative of the boundary crossing probability with respect to the initial value of the process, as the initial value approaches the boundary. In Lemma~\ref{MG_rep} we obtain a martingale representation for the indicator of the boundary non-crossing event.

The paper is structured as follows. In Section~\ref{notation} we list our assumptions and notation. In Section~\ref{main_results} we state the main results. Sections~\ref{aux_results} and~\ref{proof_main_results} contain proofs. In Section~\ref{examples} we illustrate our results in the Brownian motion, Brownian bridge and hyperbolic diffusion cases.

\section{Assumptions and notation}\label{notation}

Let $W := \{W_t:t\geq 0\}$ be the standard Brownian process on a stochastic basis $(\Omega, \mathcal{F},\{\mathcal{F}_t\}_{t\geq 0},\Px).$ Here $\{\mathcal{F}_t\}_{t\geq 0}$ is assumed to be the natural filtration of $W$, augmented in the usual sense (see e.g.\ \S\,II.67.4 in~\cite{RogersWilliamsI}). For $T>0,$ $(s,x)\in[0,T]\times \mathbb{R},$ we consider one-dimensional diffusion processes  $\X^{s,x} := \{\X_u^{s,x} : u \in [s,T]\}$ governed by the following stochastic differential equation:
\begin{equation}\label{SDE_X}
	\X_u^{s,x} = x + \int_s^u \muX(t,\X_t^{s,x})\,dt + \int_s^u\sigmaX(t,\X_t^{s,x})\,dW_t, \quad u \in [s,\T],
\end{equation}
where we assume that $\muX :[0,\T]\times \mathbb{R} \to \mathbb{R}$ and $\sigmaX : [0,\T]\times \mathbb{R} \to (0,\infty)$ satisfy the following condition sufficient for the existence and uniqueness of a strong solution to \eqref{SDE_X} (see e.g.\ p.~297--299 in~\cite{EthierKurtz1986}).
\begin{enumerate}
	\item [(C1)] The functions $\muX$ and $\sigmaX$ are continuous in both variables and such that, for some $K < \infty,$
\begin{equation*}
	x\,\muX(t,x) \leq K(1 + x^2),\quad \sigmaX^2(t,x) \leq K(1+x^2)
\end{equation*}	
for all $(t,x)\in [0,\T] \times \mathbb{R}$ and, for any $a>0,$ there exists a $K_a <\infty$ such that
	\begin{equation*}
		\lvert \muX(t,x) - \muX(t,y)\rvert + \lvert \sigmaX^2(t,x) - \sigmaX^2(t,y) \rvert \leq K_a\lvert x-y\rvert
	\end{equation*}
for all $t\in[0,\T],$ $x,y \in (-a,a).$
\end{enumerate}
We will further require the following assumptions:
\begin{enumerate}
	\item [(C2)] There exists  a $\usigma >0$ such that
	\begin{equation*}
	\inf_{(t,x) \in [0,\T]\times \mathbb{R}}\sigmaX(t,x) \geq \usigma.
	\end{equation*}
	\item [(C3)] The functions $\mu$ and $\sigma$ are bounded in $[0,T]\times \mathbb{R}.$ Furthermore there exists an $\alpha \in (0,1)$ such that $\muX,$ $\partial_{x}\muX,$ $\sigmaX,$ $\partial_t \sigmaX,$ $\partial_x \sigmaX,$  $\partial_{x}\partial_{t} \sigmaX,$ $\partial_{x}^2\sigmaX$ are H\"older continuous in $(t,x) \in [0,\T]\times \mathbb{R}$ with H\"older exponents $\alpha$ in $x$ and $\alpha/2$ in $t.$
\end{enumerate}

Let $\contSpace := \contSpace([0,\T])$ denote the Banach space of real-valued continuous functions on $[0,\T]$ endowed with the uniform norm $\| \cdot \|_{\infty},$ and let $\contSpace_0 := \{f \in \contSpace : f(0) =0 \}$. Denote by~$\cmSpace$ the Cameron--Martin space of absolutely continuous functions with square-integrable derivatives:
\begin{equation*}
 \cmSpace := \left\{h \in \contSpace_0: h(t) = \mbox{$\int_0^t$}\ell(s)\,ds ,\ t\in [0,\T],\ \ell\in \LpSpace^2([0,\T],dt)\right\}.
\end{equation*}
Denoting the derivative of~$h$ by~$\dot{h}$ (recall that $h \in \cmSpace$ are pointwise differentiable almost everywhere, see e.g.\ Corollary 12 in Chapter~5 in~\cite{Royden1988}; likewise, in what follows the dot will also denote the partial derivative $\partial_s$ with respect to the time variable~$s$), $\cmSpace$ is a Hilbert space when equipped with the inner product
\begin{equation*}
	\langle h_1,h_2\rangle_{\cmSpace} := \int_0^{\T}\dot{h}_1(s)\dot{h}_2(s)\,ds,\quad h_1, h_2\in \cmSpace.
\end{equation*}
We set $\| h\|_{\cmSpace} := \sqrt{\langle h,h\rangle_{\cmSpace}}$.

Let $\contSpace^2:= \contSpace^2([0,T])$ denote the space of twice continuously differentiable functions on $[0,\T].$ For $b\in \contSpace,$ set
\begin{equation*}
\boundarySet_b :=
		\{f \in \contSpace: f(t) < b(t),\ t\in[0,\T]\}.
\end{equation*}
The main objective of this paper is to investigate the analytic properties of the boundary non-crossing probability functional $F:  \contSpace \to \mathbb{R}$ defined by
\begin{equation*}\label{defn:F}
 F(b):= \Px(\X \in \boundarySet_b) ,\quad b \in  \{ f \in \contSpace :f(0) > x_0\},
\end{equation*}
where $\X := \X^{0,x_0}$ for a fixed $x_0 \in \mathbb{R}.$

Throughout the paper, we will also assume that the following condition concerning the ``main boundary''~$g$ is met:
\begin{itemize}
	\item [(C4)] $g \in \contSpace^2,$ $g(0) > x_0.$
\end{itemize}

\section{Main results}\label{main_results}

For a fixed boundary~$g$ satisfying (C4), consider the family of domains
\begin{equation*}
	D_t := \{(s,x) : s\in (0,t),\ x < g(s)\}\subset \mathbb{R}^2,\quad t\in(0,\T],
\end{equation*}
and the first hitting times
\begin{equation*}
	\tauX^{s,x} :=  \inf\{t \geq s : (t,\X_t^{s,x}) \in \partial \D\},\quad (s,x) \in \D, \quad \tauX := \tauX^{0,x_0}.
\end{equation*}
For a domain $U \subseteq \mathbb{R}^2,$ denote by $\contSpace^{1,2}(U)$ the space of real-valued functions $f(s,x),$ $(s,x)\in U,$ whose derivatives
\begin{equation*}
	\dot{f} := \mbox{$\frac{\partial f}{\partial s}$},\quad f' := \mbox{$\frac{\partial f}{\partial x}$}, \quad f'':= \mbox{$\frac{\partial^2 f}{\partial x^2}$}
\end{equation*}
exist and are continuous in $U$. Let
\begin{equation}\label{defn:v}
	\V(s,x) := \Px(\tauX^{s,x} = T),\quad (s,x) \in \cl(\D),
\end{equation}
denote the probability that the process \eqref{SDE_X} does not hit the boundary~$g$ prior to time~$\T.$ In Lemma~\ref{BKE} below we formally prove that~$v$ belongs to $\contSpace^{1,2}(\D),$ Remark~\ref{remark:history_PDE} after Lemma~\ref{BKE} providing historical comments on the question on the existence of these derivatives.

The following function:
\begin{equation}\label{defn:psi}
	\psi(t) := -\Ex \,\V'(t,X_t)\bm{1}_{\tau > t},\quad t\in [0,\T),
\end{equation}
plays a central role in our analysis. It turns out that the analysis of this function can be simplified by employing  Doob's $h$-transformed measure~$\mathbf{Q}$ that is introduced in Proposition~\ref{doob_transform} below. This transform uses the function  $-v'(s,x)$ (see~\eqref{defn:h}).


For $(t,y)\in \D, $ let
\begin{equation}\label{defn:h}
\hd(t,y) := -v'(t,y),\quad \gamma(t,y) := (\sigmaX^2\hd'/\hd)(t,y),
\end{equation}
and set
\begin{equation}\label{defn:Z}
	\hD_t := \frac{\hd(t\wedge \tau,\X_{t\wedge \tau})}{\hd(0,x_0)},\quad t\in [0,\T).
\end{equation}

\begin{proposition}
\label{doob_transform}
Under conditions \textnormal{(C1)--(C4)}, there exists a unique probability measure $\mathbf{Q}$ on $\mathcal{F}_T$ such that
\begin{equation*}\label{dQdP_t}
	\mathbf{Q}(A) = \Ex \,N_t\bm{1}_{A},\quad A \in \mathcal{F}_t,\quad t\in [0,\T).
\end{equation*}
Moreover, there exists a $(\mathbf{Q},\mathcal{F}_t)$-Brownian motion $\widetilde{W}$ such that $X$ satisfies the following relation for $t\in[0,\T)$\textnormal{:}
\begin{equation}\label{lemma:stopped_X}
	\X_{t\wedge \tau} = x_0 + \int_0^{t\wedge \tau}(\muX + \gamma)(s,\X_s)\,ds + \int_0^{t\wedge \tau}\sigmaX(s,\X_s)\,d\widetilde{W}_s.
\end{equation}
\end{proposition}

Using the Doob $h$-transformed measure~$\mathbf{Q},$ we obtain a simple representation for the function~$\psi$ in the following theorem.

\begin{theorem}
\label{thm:Z}
Assume that conditions \textnormal{(C1)--(C4)} are satisfied. Then the expectation in \eqref{defn:psi} is finite and the function $\psi$ admits the following representation\textnormal{:}
	\begin{equation}\label{eqn:psi_rep}
		\psi(t) = -v'(0,x_0)\mathbf{Q}(\tau >t),\quad t \in [0,\T).
	\end{equation}
Moreover, $\psi(\T-) = \mathbf{Q}(\tauX \geq T)  =0,$  $\psi$ is continuously differentiable in $(0,T),$ and
\begin{equation}\label{eqn:psi_dot_1}
	\dot{\psi}(t) = v'(0,x_0)f_{\tau}^{\mathbf{Q}}(t), \quad t\in(0,\T),
\end{equation}
where $f_{\tau}^{\mathbf{Q}}(t):= \frac{d}{dt}\mathbf{Q}(\tau \leq t)$ is the $\mathbf{Q}$-density of $\tau$.
\end{theorem}

The next theorem provides a useful representation for $f_{\tau}^{\mathbf{Q}}$. Note that, generally speaking,  the density~$f_{\tau}$ is more accessible than~$f_{\tau}^{\mathbf{Q}}$.

\begin{theorem}
\label{thm:factorisation}
Under conditions \textnormal{(C1)--(C4)},
\begin{equation}\label{eqn:doob_h}
	f_{\tau}^{\mathbf{Q}}(t) = \frac{v'(t,g(t))}{v'(0,x_0)}f_{\tau}(t), \quad t\in (0,T),
\end{equation}
where $f_{\tau}(t) := \frac{d}{dt}\Px(\tau \leq t)$ exists and is continuous.
\end{theorem}

\begin{remark}
Relation \eqref{eqn:doob_h} is similar to the assertion of Corollary 4 in~\cite{Borovkov2010}, which states that
\begin{equation*}
	f_{\tau^{\circ}}(t) = \frac{p(T-t,g(t),y)}{p(T,x,y)}f_{\tau}(t), \quad t\in (0,\T),
\end{equation*}
where $f_{\tau^{\circ}}(t)$ is the density of the first hitting time~$\tau^{\circ}$ of the boundary~$g$ for the diffusion bridge in the time-homogeneous case where the process~$X$ obeys $dX_t = \mu(X_t)\,dt + dW_t$ and is pinned at $X_0 =x,$ $X_T= y,$ $p(t,x',y')$ denoting the transition density of~$X.$ Indeed, note that $p(\,\cdot\,,\cdot\,,y)$ satisfies the backward Kolmogorov equation and hence can also be used to construct a Doob $h$-transform.
\end{remark}

The following theorem establishes the G\^ateaux differentiability of~$F$ and an integral representation for the weak derivative.

\begin{theorem}\label{thm:dF}
Let conditions \textnormal{(C1)--(C4)} be satisfied and $h\in \cmSpace \cup \contSpace^2.$ Then there exists the G\^ateaux derivative of $F$ at $g$ in the direction $h$ given by
\begin{equation}\label{riesz_markov_rep}
	\nabla_h F(g) := \lim_{\delta \to 0 } \frac{F(g + \delta h) - F(g)}{\delta}  = -\int_0^{\T} h(t)\dot{\psi}(t)\,dt,
\end{equation}
where
\begin{equation}\label{eqn:psi_dot_2}
	\dot{\psi}(t) = v'(t,g(t))f_{\tau}(t).
\end{equation}
\end{theorem}
It follows immediately from Theorems \ref{thm:Z}, \ref{thm:factorisation} and \ref{thm:dF} that the following result holds true.

\begin{corollary}\label{corollary_prob}
Under conditions \textnormal{(C1)--(C4)},
\begin{equation}\label{eqn:VL}
	\nabla_h F(g)= -v'(0,x_0)\Ex ^{\mathbf{Q}}h(\tau),\quad h \in \cmSpace \cup \contSpace^2,
\end{equation}
where $\Ex ^{\mathbf{Q}}$ denotes the expectation under the measure $\mathbf{Q}.$	
\end{corollary}

Observe that the effect of the value of the direction $h$ at time $t\in (0,T)$ on the derivative $\nabla_hF$ is proportional to the value of the first hitting time $\tau$ density $f_{\tau}^{\mathbf{Q}}(t)$ under measure $\mathbf{Q}.$ For further discussion of this observation in the special cases of the Brownian motion and Brownian bridge processes and linear boundary~$g,$ see Examples~\ref{exam_1} and~\ref{exam_2} in Section~\ref{examples}.

{The following theorem shows that $F$ is in fact $\cmSpace$-differentiable (see e.g.~\cite[p.~323]{Ikeda1981}), i.e.\ it is Fr\'echet differentiable when perturbations are restricted to space~$\cmSpace.$
\begin{theorem}\label{thm:dF_f}
Let conditions \textnormal{(C1)--(C4)} be satisfied. Then, as  $\varepsilon \to 0,$
\begin{equation*}
	\sup_{h \in \cmSpace,\, \| h\|_{ \cmSpace} = 1}\lvert F(g + \varepsilon h) - F(g) - \varepsilon\nabla_h F(g) \rvert = o(\varepsilon).
\end{equation*}
\end{theorem}
}

Our last main result provides a partial explanation to the empirically observed fast  convergence rate in the numerical method for computing $F(g)$ proposed and analysed in~\cite{Liang2023}. For $g\in\contSpace^2$ and $n \in \mathbb{N},$ let~$g_n$ denote the piecewise linear approximation of~$g$ with nodes at $(kT/n,g(kT/n))$, $k=0,1\ldots,n.$ Denote by $F_n(g_n)$ the Brownian bridge--interpolated Markov chain approximation of $F(g_n)$ that was constructed in~\cite{Liang2023}. Numerical experimentation reported in Section 5 of~\cite{Liang2023} indicated that the error $\lvert F_n(g_n) - F(g)\rvert$ behaves as $cn^{-2}$ as $n\to\infty$ for some $c>0.$ To explain this phenomenon, the following decomposition was used:
\begin{equation}\label{eqn:error_decomp}
	F_n(g_n) - F(g) = (F_n(g_n) - F(g_n)) + (F(g_n) - F(g)).
\end{equation}
For the first term on the right-hand side, it was heuristically shown in Section~3 of~\cite{Liang2023} that $n^2(F_n(g_n) - F(g_n))$ converges to a non-zero quantity. As for the second term, putting
\begin{equation*}
	h_n(t) := n^2(g_n(t) - g(t)), \quad  t \in [0,T],
\end{equation*}
one has
\begin{equation*}
	F(g_n) - F(g) = F(g  + n^{-2}h_n) - F(g).
\end{equation*}
If~$h_n$ converged to some~$h$ in $\cmSpace,$ one would expect the last difference to behave as $n^{-2}\nabla_h F(g).$ However, $h_n$ does not converge even in $\contSpace$ (cf.\ Fig.~\ref{fig:ex4}). Nevertheless, due to the integral representation~\eqref{riesz_markov_rep} for the derivative of $F(g),$ differentiation in direction $h_n$ becomes in the limit equivalent to that in direction $T^2\ddot{g}/12.$
\begin{theorem}\label{thm:dF_n}
Let conditions \textnormal{(C1)--(C4)} be satisfied, $g \in \contSpace^2(0,T)$ and $\| \ddot{g}\|_{\infty} <\infty$. Then, setting $h  (t):= \frac{T^2}{12}\ddot{g}(t),$ one has
\begin{equation*}
	\lim_{n\to \infty} \nabla_{h_n} F(g)  = \nabla_{h}F(g).
\end{equation*}
\end{theorem}
This result is illustrated by Example~\ref{exam_4} below.

\section{Auxiliary results}\label{aux_results}

To simplify our analysis, we first reduce the time-dependent boundary-crossing problem to a level-crossing problem by subtracting the boundary function from the process. For $(s,x) \in [0,T]\times \mathbb{R},$ set
\begin{equation*}\label{defn:muY}
	\muY(s,x) := \muX(s,x + g(s)) - \dot{g}(s),\quad \sigmaY(s,x) := \sigmaX(s,x + g(s)).
\end{equation*}
Since $g \in \boundarySpace,$ it follows from (C3) that the new drift and diffusion coefficients~$\muY$ and~$\sigmaY$ also satisfy conditions (C1)--(C3). Denote by $\Y^{s,x} := \{\Y_u^{s,x} : u \in [s,\T]\}$ the diffusion process specified by
\begin{equation*}
	\Y_u^{s,x} = x + \int_s^u \muY(t,\Y_t^{s,x}) \,dt +\int_s^u \sigmaY(t,\Y_t^{s,x})\,dW_t,\quad u\in[s,\T].
\end{equation*}
Clearly,
\begin{equation}\label{defn:Y_sx}
	Y_t^{s,x-g(s)} = X_t^{s,x} - g(t),\quad t\in [0,T],
\end{equation}
and
\begin{equation*}\label{defn:Y}
	\Y_t := \X_t - g(t),\quad t\in [0,\T],
\end{equation*}
coincides\ with $\Y_t^{0,y_0},$ where $y_0 := x_0 - g(0).$ For $(s,x)\in \Q:= (0,T)\times (-\infty,0)$ let
\begin{equation*}
	\tauY^{s,x} :=  \inf\{t \geq s: (t,\Y_t^{s,x})\in \partial\Q\}, \quad \tauY := \tauY^{0,y_0}.
\end{equation*}
We denote the probability of $\Y^{s,x}$ not hitting zero by time $\T$ by
\begin{equation*}\label{defn:vb}
	\bv(s,x)  := \Px(\tauY^{s,x} =\T), \quad (s,x) \in \Q.
\end{equation*}
From \eqref{defn:Y_sx}, we have
\begin{equation}\label{change_of_variable}
	\V(s,x) =  \bv(s,x-g(s)), \quad(s,x)\in \D.
\end{equation}

For $f \in \contSpace^{1,2}(\mathbb{R}^2),$ denote by $\heatX$ the differential operator associated with the process~$\X$:
\begin{align*}
	(\heatX f)(s,x) &:= (\dot{f} + \muX f'  + \mbox{$\frac{1}{2}$}\sigmaX^2f'')(s,x).
\end{align*}
The main reason for including the following lemma  is that, although the claimed existence and uniform boundedness of the derivatives are often taken for granted (or even explicitly assumed as in~\cite{Khintchine1933}), it is difficult to find a formal proof of this fact.

\begin{lemma}
\label{BKE}
The function $\V$ defined in \eqref{defn:v} belongs to $\contSpace^{1,2}(\D)\cap \contSpace(\cl(D_{T-\varepsilon}))$ for all $\varepsilon >0.$ Furthermore, it is a solution to the following boundary value problem\rm{:}
\begin{equation}\label{PDE}
\begin{cases}
	\heatX \V = 0, & \text{in } \D,\\
	\V(s, x) = \bm{1}_{s=T},  & (s,x) \in \partial \D\backslash \{s = 0\}.
\end{cases}
\end{equation}

\end{lemma}

\begin{remark}\label{remark:history_PDE}
In the case of two continuously differentiable boundaries  $g_\pm,$ such that $x_0\in (g_- (0), g_+(0)),$ and  a function  $f \in C^2([g_-(T),g_+(T)] )$ with  $f(g_{\pm}(T)) = f'(g_{\pm}(T))= f''(g_{\pm}(T)) =0,$ Theorem~3 on p.~428 in~\cite{GikhmannSkorokhod1964} states that the function
\begin{equation*}
v_f(s,x) := \Ex\, f(X_T^{s,x})\mathbf{1}_{g_-(t) < X_t^{s,x} < g_+(t),
 \ t \in [s,T]}
\end{equation*}
satisfies the Kolmogorov backward equation inside the strip between the boundaries.


However, in the case of the boundary crossing probability, the terminal function $f(x) = \mathbf{1}_{x < g(\T)}$ has a discontinuity at the boundary. Theorem~2 on p.~173 in~\cite{GikhmannSkorokhod1972} implies that our~$v$ from \eqref{defn:v} is a weak solution to the Kolmogorov backward equation with Dirichlet boundary conditions, but it was not proved that~$v$ is a classical solution. It was shown in Theorem~3 of~\cite{Lai1973} that, for a parabolic function~$f$ (i.e.\ a function which satisfies a space-time mean value property), one has $f\in \contSpace^{1,2}(\D)$ and $\heatX f=0$ in~$\D.$ By the strong Markov property of~$X,$ it is immediate that~$\V$ is parabolic and hence the claim of our Lemma~\ref{BKE} follows. However, we present here an alternative proof of that statement using the Green function of \eqref{PDE}, which allows us to bound the spatial derivatives of~$v$ at the boundary. In the case of when the coefficients of~$\heatX$ are bounded, smooth and time-homogeneous, it is know that, for any  $t\in(0,\T),$ the function $v(t,\,\cdot\,) $ is  infinitely differentiable, and all its derivatives are continuous and bounded up to $g(t)$ (see e.g.\ \cite{Cattiaux1991,Gobet2004}). Note also that the authors in~\cite{Delarue2013} prove that, in the case of a unit-diffusion coefficient and time-homogeneous drift, the taboo transition density satisfies the Kolmogorov forward equation with Dirichlet boundary conditions using analytic properties of the Green's function.
\end{remark}

\begin{proof}[Proof of Lemma \ref{BKE}]
Using standard results on Green's functions, we will first show that $\bv\in \contSpace^{1,2}(\Q).$   Due to relation \eqref{change_of_variable} and the differentiability of $g$, this will imply that $\V\in \contSpace^{1,2}(\D).$

For $f \in \contSpace^{1,2}( \mathbb{R}^2),$ define
\begin{equation}\label{defn:A}
	(\A f)(r,x):= \dot{f}(r,x) - \muY(\T-r,x)f'(r,x) - \mbox{$\frac{1}{2}$}\sigmaY^2(\T - r, x)f''(r,x).
\end{equation}
For a non-negative compactly supported $\phi \in \contSpace(\Q),$ consider the following initial-boundary-value problem on $\Q$:
\begin{equation}\label{PDE_rect}
	\begin{cases}
		(\A w -\phi)(r,x) = 0 , & (r,x) \in \Q, \\
		w(r,x) = 0,  & (r,x) \in \partial \Q \backslash\{r = \T\}.
	\end{cases}
\end{equation}
Since the diffusion coefficient is bounded away from zero due to condition (C2) and   the coefficients of~$\A$ are differentiable due to~(C3), Theorem 1.10 in Chapter~VI in~\cite{GarroniMenaldi1992} states that there exists a unique Green's function
\[
G(\,\cdot\,,\cdot\,; t,y) \in \contSpace^{1,2}((t, T)\times (-\infty,0)) \cap \contSpace((t,\T]\times (-\infty ,0]),\quad (t,y) \in D_t \cup (\{0\}\times  (-\infty,0)),
\]
for problem \eqref{PDE_rect}. Hence the unique solution to \eqref{PDE_rect} is given by the domain potential (see (2.16) in Chapter~IV in~\cite{GarroniMenaldi1992}) of the form
\begin{equation}\label{defn:domain_potential}
	w(r,x) = \int_0^r \int_{-\infty}^0\phi(t,y) G(r,x;t,y)\,dy\,dt, \quad (r,x) \in \Q,
	\end{equation}
which is bounded and continuous on $\cl(\Q)$ and belongs to $\contSpace^{1,2}(\Q).$

For $f\in \contSpace^{1,2}(\mathbb{R}^2),$ setting
\begin{equation*}
	(\heatY f)(s,x):= (\dot{f} + \muY f'  + \mbox{$\frac{1}{2}$}\sigmaY^2f'')(s,x),
\end{equation*}
it follows from \eqref{PDE_rect} that the function $u(s,x) := w(\T-s,x)$ satisfies the equation
\begin{equation}\label{eqn:Lu=-phi}
	(\heatY u)(s,x) = -(\A w)(T-s,x) =  -\phi(T-s,x),\quad (s,x) \in \Q,
\end{equation}
and from \eqref{defn:domain_potential} that admits representation
\begin{equation}\label{PDE_rep}
	u(s,x) = \int_s^{T}\int_{-\infty}^0 \phi(T - l, y)G(T-s, x; T - l, y)\,dy\,dl,
\end{equation}
for $(s,x) \in \Q.$

Now we proceed to show that the second factor in the integrand \eqref{PDE_rep} is the taboo transition density of the process~$\Y$: for $(s,x)\in \Q,$ $(t,y) \in (0,\T]\times (-\infty,0) ,$ $s<t,$
\begin{equation}\label{lemma:G_q}
G(T-s,x;T-t,y) = \qY(s,x;t,y):= \mbox{$\frac{\partial}{\partial y}$}\Px(\Y_t^{s,x} \leq y; \tauY^{s,x} > t).
\end{equation}
As the Green's function possesses the desired differentiability properties, it will eventually follow that $\bv \in \contSpace^{1,2}(\Q)$ and hence that  $\V \in \contSpace^{1,2}(\D).$

Since $u \in \contSpace^{1,2}(\Q),$ we can apply It\^o's formula to obtain, for $(s,x) \in \Q,$
\begin{equation*}
	u( \tauY^{s,x}, \Y_{ \tauY^{s,x}}^{s,x})
	 = u(s,x) + \int_s^{\tauY^{s,x}} (\heatY u)(t,\Y_t^{s,x})\,dt + \int_s^{ \tauY^{s,x}}( u' \, \sigmaY)(t,\Y_t^{s,x})\,dW_t.
\end{equation*}
Using \eqref{eqn:Lu=-phi} and $u(\tauY^{s,x}, \Y_{\tauY^{s,x}}^{s,x}) = 0,$ it follows that
\begin{equation*}
	u(s,x) = \int_s^{\tauY^{s,x}}\phi(T-t, \Y_t^{s,x})\,dt  - \int_s^{\tauY^{s,x}}( u' \, \sigmaY)(t,\Y_t^{s,x})\,dW_t.
\end{equation*}
As $u$ and $\phi$ are bounded, it follows that the It\^o integral on the right-hand side is a bounded zero-mean martingale. Hence we can take the expectations of both sides to get
\begin{equation*}
	u(s,x) = \Ex \int_s^{T} \phi(T-t, \Y_t^{s,x})\bm{1}_{\tauY^{s,x} \geq t}\,dt .
\end{equation*}
Since $\phi \geq 0$ is bounded, by the Fubini--Tonelli theorem we can interchange the order of integrals to get
\begin{equation}\label{prob_rep}
	u(s,x) = \int_s^{T}\int_{-\infty}^{0} \phi(T-t, y) \Px(\Y_t^{s,x} \in dy ,\tauY^{s,x} > t)\,dt .
\end{equation}
Equating \eqref{PDE_rep} and \eqref{prob_rep}, we  use the functional monotone class theorem (see e.g.\ Theorem~4 in Chapter~1, \S~5 in~\cite{Krylov1995}) to obtain that the following relation holds for any Borel set $A \in \mathcal{B}((-\infty,0]) $ and $(s,x)\in \Q,$ $(t,y) \in (0,\T]\times (-\infty,0) ,$ $s<t$:
\begin{equation*}
	\int_A G(T-s,x;T-t,y)\,dy = \Px(\Y_t^{s,x} \in A, \tauY^{s,x} > t).
\end{equation*}
This implies that the measure $\Px(\Y_t^{s,x} \in \boldsymbol{\cdot} \, , \tauY^{s,x}>t)$ admits the density given by the left-hand side of \eqref{lemma:G_q}. Hence, for $(s,x) \in \Q,$
\begin{align}
\label{PDE_rep_U}
	\bv(s,x) = \Px(Y_T^{s,x} \leq 0,\tauY^{s,x}=T) = \int_{-\infty}^{0}G(T-s, x; 0, y)\,dy.
\end{align}
That $\bv\in \contSpace^{1,2}(\Q)$  follows now from differentiability of $G$ and the ``Gaussian bounds'' for the derivatives of $G(r,x;t,y)$ with respect to~$r$ and~$x$ from Theorem~1.10 in Chapter~VI in~\cite{GarroniMenaldi1992}. Indeed, using these bounds, one verifies that, when differentiating the right-hand side of~\eqref{PDE_rep_U}, one can  interchange the order of differentiation and integration and conclude the claimed  properties from the respective order continuous differentiability of~$G.$
Using the chain-rule and \eqref{change_of_variable}, it follows that $\V \in \contSpace^{1,2}(\D)$ as well.

Once the desired differentiability of $\V$ has been established, relation \eqref{PDE} is proved using a standard argument based on It\^o's formula: for $t\in[0,\T],$
\begin{equation}
\label{v_MG}
	\V(t \wedge \tauX, \X_{t\wedge \tauX}) = \V(0,X_0)  + \int_0^{t\wedge \tauX} (\heatX \V)(s,X_s)\,ds + \int_0^{t\wedge \tauX} (  \V' \, \sigmaX)(s, \X_s)\,dW_s.
\end{equation}
Since $\V$ is bounded and $\X$ is a strong Markov process, $\V(t\wedge \tauX, \X_{t\wedge \tauX}) =\Ex [\mathbf{1}_{\tauX =\T}\,|\,\mathcal{F}_{t\wedge \tauX}]  $ is a bounded Doob martingale. Since the It\^o integral on the right-hand side of~\eqref{v_MG} is a local martingale, it follows that $\int_0^{t \wedge \tauX} (\heatX \V)(s, \X_s)\,ds$ is a local martingale as well. As the latter process is also continuous and has zero mean and finite variation, it must be equal to zero (see e.g.\ p.~120 in~\cite{RevuzYor}), so that $\heatX \V = 0$ in~$\D.$ Lemma~\ref{BKE} is proved.
\end{proof}

In the next lemma we establish important differentiability properties of $\V$ that we will need in what follows.

\begin{lemma}\label{lemma:boundary_grad}
For all $s\in (0,\T),$ the limits
\begin{equation*}
	\V'(s, g(s)) := \lim_{x\uparrow g(s)}\V'(s,x),\quad \V''(s, g(s)) := \lim_{x\uparrow g(s)}\V''(s,x)
\end{equation*}
exist and are finite. Moreover, $\V'(s,g(s))$ and $\V''(s,g(s))$ are continuous functions of $s \in (0,\T).$ Furthermore, for any $\varepsilon \in (0,T),$
\begin{equation*}
	\sup_{(s,x) \in D_{\T-\varepsilon}} \lvert v'(s,x)\rvert <\infty,\quad \sup_{(s,x) \in D_{\T-\varepsilon}} \lvert v''(s,x)\rvert <\infty.
\end{equation*}
\end{lemma}

\begin{proof}
Let $v^{(1)}(s,x):= v'(s,x)$ and $v^{(2)}(s,x):= v''(s,x).$ Differentiating the right-hand side of \eqref{PDE_rep_U} and using the ``Gaussian bounds'' from Theorem~1.10 in Chapter~VI in~\cite{GarroniMenaldi1992} to interchange the order of differentiation and integration, we establish the existence of the limits
\begin{equation}
\label{bd_der_1}
	\bv^{(\ell)}(s,0) := \lim_{x\uparrow 0}\bv^{(\ell)}(s,x) = \int_{-\infty}^{0} \lim_{x\uparrow 0}\partial_x^{\ell}G(T-s,x;0,y)\,dy,\quad  \ell \in \{1,2\}.
\end{equation}
Further, the same theorem from \cite{GarroniMenaldi1992} asserts that there exist constants $c_0$ and $c_1$ that only depend on $\muY,$ $\sigmaY,$ and $\T,$ such that, for $t < s' < s $ and $(s,x),(t,y) \in \Q,$ $r:=T-s,$ $r':=T-s',$ one has
\begin{equation}\label{grad_bd_est}
	\lvert \partial_x^{\ell} G(r,x;t,y) - \partial_x^{\ell}G(r',x;t,y)  \rvert
 \leq
 c_0 \frac{(r-r')^{(\alpha + 2-\ell)/2}}{(r' - t)^{(3+\alpha)/2}}\exp\left\{ -c_1\frac{(x -y)^2}{r - t} \right\}.
\end{equation}
Using \eqref{bd_der_1} and \eqref{grad_bd_est}, we obtain
\begin{align*}
	\lvert \bv^{(\ell)}(s,0) - \bv^{(\ell)}(s',0) \rvert&\leq  \int_{-\infty}^{0}\lvert \partial_x^{\ell}G(T-s',0;0,y)- \partial_x^{\ell}G(T-s,0;0,y) \rvert \,dy\\
	&\leq c_0 \frac{(s-s')^{(\alpha + 2 -\ell)/2}}{(T-s')^{(3+\alpha)/2}} \int_{-\infty}^0\exp\left\{\frac{-c_1 y^2}{T-s'} \right\}\,dy,
\end{align*}
which converges to zero as  $|s'-s| \to 0,$  establishing the continuity of $\bv(s,0),$ $s\in (0,\T).$

The boundedness of $\sup_{(s,x) \in D_{T-\varepsilon}} \bv^{(\ell)}(s,x),$ $\ell \in \{1,2\},$ is ensured by the same ``Gaussian bounds'' theorem in~\cite{GarroniMenaldi1992}. The assertions  for~$\V$  stated in the lemma    immediately   follow now from the relation~\eqref{change_of_variable} between~$\bv$ and~$\V$. Lemma \ref{lemma:boundary_grad} is proved.
\end{proof}

Denote the unit-diffusion transform of $\Y$ by
\begin{equation}\label{defn:Psi}
	\Psi(t,y) := \int_0^y\frac{dx}{\sigmaY(t,x)}, \quad (t,y) \in [0,\T]\times (-\infty,0],
\end{equation}
and the inverse of $z=\Psi(t,y)$ with respect to $y$ by $\Psi^{-1}(t,z).$ Under assumptions (C2) and (C3), $\Psi$ belongs to $ \contSpace^{1,2}([0,T]
 \times \mathbb{R}),$ and hence it follows from It\^o's formula that $Z^{s,z}:=\big\{Z_t^{s,z}:=\Psi\big(t,\Y_t^{s,\Psi^{-1}(s,z)}\big):t\in [s,\T]\big\}$ is a diffusion process with a unit-diffusion coefficient (see e.g.\ p.~34 in~\cite{GikhmannSkorokhod1972}), i.e.
\begin{equation*}
	Z_u^{s,z} = z + \int_s^u\muZ(t, Z_t^{s,z})\,dt + W_u - W_s,\quad u\in [s,\T],
\end{equation*}
with the  drift coefficient given by
\begin{equation}\label{defn:eta}
	\muZ(t,z) := (\dot{\Psi} + \muY/\sigmaY - \mbox{$\frac{1}{2}$}\sigmaY')(t, \Psi^{-1}(t,z)).
\end{equation}
Furthermore, for $(t,z) \in [0,\T]\times \mathbb{R},$ set $\Eta(t,z) := \int_0^z\muZ(t,x)\,dx,$ and consider the continuous functional~$G_t$ on $\contSpace([0,T-t])$ defined by
\begin{equation}\label{defn:G}
		G_t(w) := \exp\bigg\{ \Eta(T,w_{T-t}) -\Eta(t,w_0)  - \frac{1}{2}\int_t^T(2\dot{\Eta} + \muZ^2+ \muZ')(u,w_{u-t})\,du \bigg\}
\end{equation}
for continuous $w := \{w_s:s\in[0,\T-t]\}.$

\begin{lemma}\label{prop:VL_BD}
Let conditions \textnormal{(C1)--(C4)} be satisfied. Then
\begin{equation}\label{v_bd_deriv_meander}
v'(t,g(t)) = -\sqrt{\frac{2}{\pi(\T-t)}} \cdot\frac{\Ex \,G_t(-W^{\oplus,\T-t})}{\sigmaX(t,g(t)) }, \quad t\in (0,\T),
\end{equation}
where $W^{\oplus,u} := \{W_s^{\oplus,u}: s\in [0,u]\}$ is the Brownian meander of length $u$ \textnormal{(}the standard Brownian motion conditioned to stay positive in the time interval $(0,u]$\textnormal{)}.
\end{lemma}

\begin{remark}
In the time-homogeneous case, it follows from Lemma~\ref{prop:VL_BD} that representation~\eqref{riesz_markov_rep} for the  G\^ateaux derivative $\nabla_hF$ coincides with the representation obtained in Theorem 5 of \cite{Borovkov2010}.
\end{remark}


\begin{proof}[Proof of Lemma~\ref{prop:VL_BD}]
Using the boundary condition $v(t,g(t))=0$ and the fact that $v'(t,\,\cdot\,)$ has a finite limit at the boundary $g(t)$ due to Lemma~\ref{lemma:boundary_grad}, it follows that
\begin{equation}\label{eqn:bd_pre_limit}
	\lim_{x\uparrow g(t)}v'(t,x) = -\lim_{\delta \downarrow 0}\frac{1}{\delta}\int_{g(t)-\delta}^{g(t)}v'(t,z)\,dz = -\lim_{\delta \downarrow 0} \delta^{-1}v(t,g(t)-\delta).
\end{equation}
We will obtain \eqref{v_bd_deriv_meander} by computing the limit on the right-hand side of \eqref{eqn:bd_pre_limit}, using a change of measure and the well-known explicit expression for the constant boundary non-crossing probability of the standard Brownian motion.

Recall our notation \eqref{defn:Psi} and, for $\delta >0,$ set
\begin{equation*}
	\Wz_u:= Z_{t+u}^{t,\Psi(t,-\delta)} - \Psi(t,-\delta) ,\quad u \in[0,\T-t].
\end{equation*}
Let $\muZ_{t,\delta}(s,x) := \muZ(t+s, x + \Psi(t,-\delta)),$ $s\in[0,\T-t],$ and note that
\begin{equation*}
	\Wz_u = \int_0^u \muZ_{t,\delta}(s, \Wz_s)\,ds + W_{t+u} - W_t,\quad u \in [0,T-t].
\end{equation*}
Introducing for $t \in (0,T) $ the event
\begin{equation*}
	E_t^{(\delta)} := \{\widetilde{W}_u < -\Psi(t,-\delta), u \in [0,T-t]\},
\end{equation*}
one has
\begin{align}\label{pre_measure_change}
    v(t,g(t)-\delta) &= \Px(\X_u^{t,g(t)-\delta} < g(u),u \in [t,T])= \Px(\Y_u^{t,-\delta} < 0, u \in [t,T])\nonumber\\
    &= \Px(Z_u^{t,\Psi(t,-\delta)} < 0, u \in [t,T])= \Px(E_t^{(\delta)}).
\end{align}
Let
\begin{equation}\label{defn:Lambda}
    \Lambda_{t}^{(\delta)}  :=\exp\left\{\int_0^{T-t} \muZ_{t,\delta}(u, \Wz_u)\,d\Wz_u - \frac{1}{2}\int_0^{T-t} \muZ_{t,\delta}^2(u, \Wz_u)\,du \right\}.
\end{equation}
Since $\muZ_{t,\delta}$ is locally bounded, assumptions~(I) and~(II) of Theorem~7.18 and conditions of Theorem~7.19 in~\cite{Lipster2001} are satisfied, which implies that  the probability measure~$\Pz$ (with the corresponding expectation~$\Ez$) with the Radon--Nikodym derivative $d\widetilde{\Px}/d\Px := 1/\Lambda_{t}^{(\delta)}$ is such that~$\Wz$ is a~$\Pz$-Brownian motion. It follows from \eqref{pre_measure_change} and \eqref{defn:Lambda} that
\begin{equation}\label{bd_deriv_pre_limit}
	v(t,g(t) - \delta) = \widetilde{\Ex }\,\Lambda_{t}^{(\delta)}\bm{1}_{E_{t}^{(\delta)}}
    =\Ez[\Lambda_{t}^{(\delta)}\,|\, E_{t}^{(\delta)}]\Pz(E_{t}^{(\delta)}),\quad t\in(0,T).
\end{equation}
Since
\begin{equation*}
	-\Psi(t,-\delta) \equiv -\int_0^{-\delta}\frac{dy}{\sigmaY(t,y)}= \frac{\delta}{\sigmaY(t,0)} +o(\delta) =\frac{\delta}{\sigmaX(t,g(t))} + o(\delta), \quad \delta \downarrow 0,
\end{equation*}
using the well-known expression for the level-crossing probability of the Brownian motion (see e.g.\ item~1.1.4 on p.~153 in~\cite{Borodin2002}), we obtain
\begin{align}\label{bcp:bm}
	\lim_{\delta \downarrow 0}\frac{\Pz(E_{t}^{(\delta)})}{\delta} &= \lim_{\delta \downarrow 0}\frac{1}{\delta}\left(1-2\Phi\left(\frac{\Psi(t,-\delta)}{\sqrt{\T-t}}\right)\right) = \lim_{\delta\downarrow 0}\left(\frac{1}{\delta}\cdot 2\frac{-\Psi(t,-\delta)}{\sqrt{T-t}}\cdot \frac{1 + o(\delta)}{\sqrt{2\pi}} \right)\nonumber \\
	&=\sqrt{\frac{2}{\pi(\T-t)\sigmaX^2(t,g(t))}}.
\end{align}

It is not hard to verify that $\Eta(t,z) = \int_{0}^z\muZ(t,x)\,dx \in \contSpace^{1,2}((0,T)\times \mathbb{R}).$ Using It\^o's formula to transform the It\^o integral in \eqref{defn:Lambda} into a Brownian functional, we get
\begin{align*}
	\int_0^{\T-t} \muZ_{t,\delta}(u, \Wz_u)\,d\Wz_u	&=\int_0^{\T-t} \muZ(t+u,\Wz_u +\Psi(t,-\delta))\,d\Wz_u \\
	&= \Eta(\T,\Wz_{\T-t}+\Psi(t,-\delta)) - \Eta(t,\Psi(t,-\delta)) \\
	&\qquad -\int_0^{\T-t} (\dot{\Eta} + \mbox{$\frac{1}{2}$}\Eta'')(t+u,\Wz_u +\Psi(t,-\delta)) \,du,
\end{align*}
and hence, as $\Eta'' = \muZ',$ one has
\begin{align*}
     \Lambda_{t}^{(\delta)} &=  \exp\biggl\{ \Eta(\T, \Wz_{T-t} + \Psi(t,-\delta) )-\Eta(t, \Psi(t,-\delta)) \nonumber\\
    &\qquad\qquad -\frac{1}{2}\int_0^{T-t} (2\dot{\Eta} + \muZ^2 + \muZ')(t+u,\Wz_u + \Psi(t,-\delta))\,du \biggr\}\\
    &= G_t(\Wz + \Psi(t,-\delta)),
\end{align*}
where $G_t$ was defined in \eqref{defn:G}. By Theorem 2.1 in~\cite{Durrett1977}, as $\delta \downarrow 0,$ the conditional distributions (under $\Pz$) of the processes $-\Wz \in C([0,\T-t])$ given $E_{t}^{(\delta)}$ converge weakly to the law of the Brownian meander $W^{\oplus,(T-t)}$ of length~$\T-t$. Using the argument on p.~123--124 in~\cite{Borovkov2010}, it follows that the family of random variables $\{\xi_t^{(\delta)} \stackrel{d}{=} (G_t(\Wz + \Psi(t,-\delta))\,|\,E_t^{(\delta)}) : \delta \in (0,1)\}$ is uniformly integrable and hence, as the functional $G_t$ is continuous, we obtain
\begin{equation}\label{lambda_brownian_functional}
	\lim_{\delta \downarrow 0}\Ex \,\xi_t^{(\delta)} = \lim_{\delta \downarrow 0}\Ez [G_t(\Wz + \Psi(t,-\delta))\,|\, E_{t}^{(\delta)}] = \Ez \,G_t(-W^{\oplus,(T-t)}).
\end{equation}
Dividing both sides of \eqref{bd_deriv_pre_limit} by $\delta$ and passing to the limit as $\delta \downarrow 0$ using \eqref{bcp:bm} and \eqref{lambda_brownian_functional}, we obtain \eqref{v_bd_deriv_meander}. Lemma~\ref{prop:VL_BD} is proved.
\end{proof}

\begin{lemma}\label{lemma:existence_taboo}
For $0\leq s<t\leq\T,$ $x < g(s),$ $y<g(t),$ there exists the  taboo transition density
\begin{equation*}
	\q(s,x;t,y) := \Px(X_t^{s,x} \in dy, \tauX^{s,x} > t)/dy
\end{equation*}
for the process~$\X$. Furthermore, for all $(s,x),$ the function $q(s,x;t,y)$ is continuously differentiable with respect to $y$ and there exist finite limits $\lim_{y\uparrow g(t)}\partial_yq(s,x;t,y).$
\end{lemma}

\begin{proof}
Using \eqref{lemma:G_q}, for $0\leq s<t\leq\T,$ $x < g(s),$ $y<g(t),$ one has
\begin{align*}
	\Px(\X_t^{s,x} \leq y ; \tauX^{s,x} >  t) &= \Px(\Y_t^{s,x-g(s)} + g(t)\leq y; \tauY^{s,x-g(s)} >t )\\
		&= \int_{-\infty}^{y-g(t)} G(T-s,x-g(s); T-t,  z)\,dz.
\end{align*}
Therefore, the desired taboo transition density exists and is given by
\begin{equation}\label{greens_q}
	\q(s,x;t,y)  = G(T-s,x - g(s) ;T-t,y - g(t) ).
\end{equation}
We will now proceed to prove the second claim of the lemma.   Let $\A^*$ be the adjoint of the operator $\A$ defined in \eqref{defn:A}: for compactly supported $f \in \contSpace^{1,2}(\Q),$
\begin{equation*}
	(\A^*f)(r, y) = \partial_y[\muY(T-r,y)f(r,y)] - \mbox{$\frac{1}{2}$}\partial_{yy}[\sigmaY^2(T-r,y)f(r,y)] - (\partial_rf)(r,y),
\end{equation*}
see e.g.\ Section 1.8 in~\cite{Friedman1964}. Consider the adjoint problem to \eqref{PDE_rect}: for a non-negative compactly supported $\phi \in \contSpace(\Q),$
\begin{equation}\label{PDE_rect_adjoint}
	\begin{cases}
		(\A^* w  - \phi )(r,y)= 0 , & (r,y) \in \Q, \\
		w(r,y) = 0,  & (r,y) \in \partial \Q\backslash \{r=\T\}.
	\end{cases}
\end{equation}
Using the product rule, it follows from the differentiability of $\muY$ and $\sigmaY$ (ensured by condition (C3)) that the coefficients of $\A^*$ satisfy the conditions in Theorem~17 from Chapter~2, \S7 in~\cite{Friedman1964}. Hence, for $(s,x),\,(t,y)\in \Q,$ $s<t,$ one has
\begin{equation}\label{adjoint_greens_func}
	G(t,y;s,x) = G^*(s,x;t,y),
\end{equation}
where $G^*$ is the Green's function for the adjoint problem \eqref{PDE_rect_adjoint}. Using relations \eqref{greens_q} and \eqref{adjoint_greens_func}, the following equality holds for $0\leq s<t\leq\T,$ $x < g(s),$ $y<g(t):$
\begin{equation}\label{forward_greens_function}
	q(s,x; t, y) = G^*(T-t, y-g(t); T-s,x-g(s)).
\end{equation}
Using the ``Gaussian bounds'' for the derivatives of $G^*(s,x;t,y)$ with respect to~$s$ and~$x$ from Theorem~1.10 in Chapter.~VI in~\cite{GarroniMenaldi1992},
it follows that $\partial_yG^*(T-t,y;T-s,x)$ is bounded for all $y\leq 0.$ Hence \eqref{forward_greens_function} implies that $\partial_y\q(s,x;t,y)$ exists and is continuous up to the boundary. Lemma \ref{lemma:existence_taboo} is proved.
\end{proof}

\begin{lemma}\label{lemma:fpt_cont}
The distribution of the boundary hitting time $\tau = \tau^{0,x_0}$ has a continuous density $f_{\tau}(t)$ on $(0,\T)$ that admits the following representation:
\begin{equation}\label{first_passage_time_density_rep}
	f_{\tau}(t) = -\mbox{$\frac{1}{2}$}\sigmaX^2(t,g(t)) \lim_{y\uparrow g(t)}\partial_yq(0,x_0;t,y),\quad t\in (0,\T).
\end{equation}
\end{lemma}

\begin{remark}
The existence of the first crossing time density for Brownian motion and a boundary function with bounded derivative was established in~\cite{Fortet1943}. In~\cite{Strassen1967} it was shown, again for the Brownian motion, that the density exists and is continuous in each open interval where the boundary is continuously differentiable (see also~\cite{Peskir2002} for a simplified proof of this result). The existence of the first crossing time density was established for a Lipschitz boundary~$g$ and a general time-homogeneous diffusion process with differentiable coefficients in~\cite{Downes2008}. In Lemma~\ref{lemma:fpt_cont} we extend the existence of the first passage time density to the case of a time-inhomogeneous diffusion with twice-differentiable coefficients and a twice-continuously differentiable boundary.
\end{remark}

\begin{proof}[Proof of Lemma~\ref{lemma:fpt_cont}]
Since $\tauX = \tauY$, it follows from~\eqref{forward_greens_function} that, for $t\in (0,\T),$
\begin{align*}
	\Px(\tauX \geq t) &= \Px(\tauY \geq t) =\int_{-\infty}^{0} \Px(\Y_t \in dy, \tauY \geq t) \,dy\\
	&= \int_{-\infty}^{0} G^*(T-t,y;T,y_0)\,dy.
\end{align*}
The last integral here is differentiable with respect to $t,$ with its derivative equal to
\begin{equation*}
	\partial_t\int_{-\infty}^{0} G^*(T-t,y;T,y_0)\,dy=- \int_{-\infty}^0 (\partial_t G^*)(T-t,y;T,y_0) \,dy,
\end{equation*}
where the interchange of the order of differentiation and integration is justified by the above-mentioned  ``Gaussian bounds" from~\cite{GarroniMenaldi1992} on the derivative of $G^*(s,x;t,y)$ with respect to $s$. Using $(\A^*G^*)(r,y) =0,$ where the operator $\A^*$ is applied to $G^*$ as a function of its first two arguments, it follows that
\begin{align*}
	f_{\tau}(t) &= -\partial_t \Px(\tau >t) =\int_{-\infty}^0 (\partial_r G^*)(r,y;T,y_0)|_{r=T-t} \,dy\\
	&= \int_{-\infty}^0 \bigl(\partial_y[\muY(T-r,y) G^*(r,y;T,y_0)] \\
	&\qquad\qquad- \mbox{$\frac{1}{2}$}\partial_{yy}[\sigmaY^2(T-r,y)G^*(r,y;T,y_0)] \bigr)\big|_{r=T-t}\,dy\\
	&= -\mbox{$\frac{1}{2}$}\sigmaY^2(t,0)\partial_{y}G^*(T-t,y;T,y_0)|_{y=0},
\end{align*}
where we used the boundary conditions $G^*(t,0;s,x)=G^*(t,-\infty;s,x) =0$ and the observation that $\lvert \partial_y\overline{\sigma}^2 \rvert$ cannot grow faster than a linear function of $y \to -\infty,$ while~$G^*$ vanishes exponentially fast due to the above-mentioned bounds from~\cite{GarroniMenaldi1992}. Now representation \eqref{first_passage_time_density_rep} follows from $\eqref{forward_greens_function}$. The continuity of $f_{\tau}(t)$ follows from the continuity of $\sigmaY(t,y)$ and $\partial_y q(0,y_0;t,y)|_{y=g(t)}$ in~$t.$
\end{proof}

\begin{lemma}\label{lemma:hypoellip}
The derivatives $\dot{v}'$ and $v'''$ exist  and are H\"older continuous in~$\D$ with  H\"older exponent $\alpha.$
\end{lemma}

The assertion of Lemma \ref{lemma:hypoellip} follows from Theorem 10 in Chapter 3 in~\cite{Friedman1964} since $\heatX \V =0$ and the coefficients of $\heatX$ satisfy our condition (C3).

\begin{lemma}\label{MG_rep}
For $(s,x) \in \D,$
	\begin{equation}\label{ind_martingale_rep}
	\mathbf{1}_{\tau^{s,x} = T} = \V(s,x) +  \int_s^{T} ( \V' \, \sigmaX)(t,X_t^{s,x})\bm{1}_{\tau^{s,x} >t}\,dW_t,\quad \Px\text{-a.s}.
	\end{equation}
\end{lemma}

\begin{remark}
In the case of the standard Brownian motion process and a flat boundary $g(t)= \text{const},$ the martingale representation in~\eqref{ind_martingale_rep} was obtained in Lemma~1 of~\cite{Shiryaev2004} using the moment generating function of~$\tau$ and It\^o's formula. This was extended to the case of time-homogeneous diffusion processes below a fixed level in~\cite{Feng2015}. Our Lemma~\ref{MG_rep} uses a different approach to extend the martingale representation to time-inhomogeneous diffusions below time-dependent boundaries, where the techniques employed in~\cite{Shiryaev2004} and~\cite{Feng2015} do not work.
\end{remark}

\begin{proof}[Proof of Lemma~\ref{MG_rep}]
Since $\V\in\contSpace^{1,2}(\D)$ and $\heatX \V =0$ by Lemma \ref{BKE}, by It\^o's formula one has
\begin{equation*}
	\V( \tauX^{s,x} , \X_{ \tau^{s,x}}^{s,x}) = \V(s,x) + \int_s^{\tauX^{s,x}}(\V' \, \sigmaX)(t,\X_t^{s,x})\,dW_t.
\end{equation*}
As $\V(\T,\X_{\T}^{s,x}) =1$ a.s.\ and $\V(\tau^{s,x},X_{\tau^{s,x}}^{s,x})  = \V(\tau^{s,x}, g(\tau^{s,x})) = 0$ on the event $\{\tau^{s,x} <T\},$ it follows that
\begin{equation*}
	\V(\tauX^{s,x} , \X_{\tau^{s,x}}^{s,x}) = \V(\T,\X_{\T}^{s,x})\mathbf{1}_{\tauX^{s,x} = \T} + \V(\tauX^{s,x},\X_{\tauX^{s,x}}^{s,x})\mathbf{1}_{\tauX^{s,x} < \T} = \mathbf{1}_{\tauX^{s,x} = \T}.
\end{equation*}
Combining this with the previous displayed formula, we obtain \eqref{ind_martingale_rep}.
\end{proof}

\section{Proofs of main results}\label{proof_main_results}

\begin{proof}[Proof of Proposition \ref{doob_transform}]
Here we are using the approach from \S\,IV.39 in~\cite{RogersWilliamsII}. Since $v' \in \contSpace^{1,2}(\D)$ due to Lemma \ref{lemma:hypoellip}, we can apply It\^o's formula to obtain
\begin{equation}\label{eqn:dh}
	v'(t\wedge \tau,\X_{t\wedge\tau}) = v'(0,X_0) + \int_0^{t\wedge\tau} (\heatX v')(s,\X_s) \,ds + \int_0^{t\wedge \tau}(\sigmaX \,v'')(s,\X_s)\,dW_s.
\end{equation}
As $\heatX v =0 $ in $\D$ by Lemma~\ref{BKE}, due to the linearity of differentiation we also have
\begin{equation*}\label{eqn:Lv'}
	\heatX v' = \heatX \hd =0 \quad \text{in } \D.	
\end{equation*}
Hence it follows from \eqref{eqn:dh} and the definition of $\hD_t$ in \eqref{defn:Z} that $\hD$ satisfies the following linear stochastic equation:
\begin{equation*}\label{linear:SDE}
	\hD_{t} =1 + \int_0^{t\wedge \tau}\frac{\gamma}{\sigmaX}(s,\X_s) \hD_s\,dW_s, \quad t\in[0,\T],
\end{equation*}
whose unique solution is given by the Dol\'eans exponential $\mathcal{E}(V)_t$ of the process
\begin{equation*}
	V_t := \int_0^{t\wedge \tau} \frac{\gamma}{\sigmaX}(s,X_s)\,dW_s,\quad t\in [0,\T].
\end{equation*}
That is, denoting by $[V]_t := \int_0^{t\wedge \tau}\frac{\gamma^2}{\sigma^2}(s,\X_s)\,ds,$ $t\in[0,\T],$ the quadratic variation process of~$V,$ we have
\begin{equation*}\label{rel:Z_MG}
	\hD_t = \mathcal{E}(V)_{t}\equiv \exp\{V_t - \mbox{$\frac{1}{2}$}[V]_t\},\quad t\in[0,\T].
\end{equation*}
Since, for any $\varepsilon \in (0,\T),$ the function $\V'$ (and hence $\hd$ as well) is bounded in $\cl(D_{\T-\varepsilon})$ due to Lemma~\ref{lemma:boundary_grad}, it follows that $\hD_{t\wedge (\T-\varepsilon)} =\mathcal{E}(V)_{t\wedge (\T-\varepsilon)}$ is a bounded martingale, and thus by the Cameron--Martin--Girsanov theorem and the Daniell--Kolmogorov extension theorem (see e.g.\ p.~82 in~\cite{RogersWilliamsII}) there exists a unique measure $\mathbf{Q}^{\varepsilon}$ on $\mathcal{F}_T$ which is absolutely continuous with respect to $\Px$ and $N_t$ is a version of the Radon--Nikodym derivative of $\mathbf{Q}^{\varepsilon}$ with respect to $\Px$ when the measures are restricted to $\mathcal{F}_t$:
\begin{equation*}\label{dQedP_t}
	 \frac{d\mathbf{Q}^{\varepsilon}}{d\Px}\biggr|_{\mathcal{F}_{t}} = \hD_{t\wedge (T-\varepsilon)}\quad \Px\text{-a.s.},\quad t\in[0,\T],
\end{equation*}
or, equivalently, $\mathbf{Q}^{\varepsilon}(A) = \Ex N_{t\wedge (\T - \varepsilon)}\bm{1}_A,$ $A \in \mathcal{F}_t.$ Moreover, under $\mathbf{Q}^{\varepsilon},$
\begin{equation*}
	\widetilde{W}_{t}^{(\varepsilon)} := W_{t\wedge \tau\wedge (T-\varepsilon)} - \int_0^{t\wedge \tau\wedge (T-\varepsilon)} \frac{\gamma}{\sigmaX}(s,X_s) \,ds ,\quad t\in [0,\T],
\end{equation*}
is a stopped $\{\mathcal{F}_t : t\in[0,\T]\}$-Brownian motion. Hence there is a measure $\mathbf{Q}$ on $\contSpace([0,\T))$ such that $\mathbf{Q}|_{\mathcal{F}_{\T-\varepsilon}} =\mathbf{Q}^{\varepsilon}|_{\mathcal{F}_{\T-\varepsilon}}$ for all $\varepsilon>0,$ (cf.\ p.\,88 in~\cite{RogersWilliamsII}) and thus relation \eqref{lemma:stopped_X} is proved.
\end{proof}

\begin{proof}[Proof of Theorem \ref{thm:Z}]
The fact that $\Ex \,v'(t,X_t)\bm{1}_{\tau >t}$ is finite for all $t\in [0,\T)$ is immediate from the uniform boundedness of $v' $ in $\cl(D_{\T-\varepsilon}),$ $\varepsilon\in (0,\T),$ established in Lemma~\ref{lemma:boundary_grad}.
Representation \eqref{eqn:psi_rep} follows directly from Proposition~\ref{doob_transform}: for $t \in [0,\T),$
\begin{equation*}
	\psi(t)=-\Ex \,v'(t,X_t)\bm{1}_{\tau >t}  =  -v'(0,x_0)\Ex\, N_t \bm{1}_{\tau >t} = -v'(0,x_0)\mathbf{Q}(\tau >t).
\end{equation*}
For the taboo transition density of $X$ under $\mathbf{Q},$ one can easily see that,
\begin{equation*}\label{defn:q_hd}
	q^{\hd}(s,x;t,y) := \mathbf{Q}(X_t^{s,x}\in dy, \tauX^{s,x} > t)/dy =  \frac{\hd(t,y)}{\hd(s,x)}q(s,x;t,y)
\end{equation*}
for $(s,x)\times(t,y)\in\D^2,$ $s < t.$ It follows from Lemmata \ref{lemma:boundary_grad} and \ref{lemma:existence_taboo} that $q^{\hd}(s,x;t,y)$ is also continuously differentiable with respect to $y$ up to the boundary $g(t)$, so we can repeat the steps in the proof of Lemma \ref{lemma:fpt_cont} to establish the existence and continuity of the first passage time density $f_{\tau}^{\mathbf{Q}}$ of $\tau,$ obtaining \eqref{eqn:psi_dot_1}.
Now we will show that $\lim_{t\uparrow \T}\psi(t)$ exists, ensuring the continuity of $\psi(t)$ in $t\in [0,\T].$ Since the taboo transition density $q$ exists due to Lemma~\ref{lemma:existence_taboo}, we have
\begin{equation*}
	\psi(t) = -\Ex \,\V'(t,X_t)\bm{1}_{\tau >t} = -\int_{-\infty}^{g(t)} \V'(t,y) q(0,x_0;t,y)\,dy.
\end{equation*}
Since $q(0,x_0;t,y)$ is continuously differentiable with respect to $y$ up to the boundary $g(t)$ by Lemma~\ref{lemma:existence_taboo}, using integration by parts and the boundary conditions $\V(t,g(t)) = 0,$ $\lim_{y\downarrow -\infty}\V(t,y)=1,$ and that $\lim_{y\downarrow -\infty}q(0,x_0;t,y)=0,$ we get
\begin{equation*}
	\psi(t) =  \int_{-\infty}^{g(t)} \V(t,y)\partial_yq(0,x_0;t,y)\,dy.
\end{equation*}
Using the change of variable $z + g(t) = y$ and relation \eqref{forward_greens_function}, we have
\begin{equation}\label{psi_integ_parts}
	\psi(t)
	= \int_{-\infty}^{0}\V(t,z + g(t))\partial_y G^*(T-t,z;T,x_0-g(0))\,dz.
\end{equation}
Since $0\leq v \leq 1,$ it follows from Theorem 1.10 in Chapter~VI in \cite{GarroniMenaldi1992} that there exists constants $C$ and $c_0$ which only depend on $\muY$ and $\sigmaY$ such that the absolute value of the integrand in \eqref{psi_integ_parts} is bounded by
\begin{equation*}
	 \frac{C}{t}\exp\left\{-\frac{c_0}{t}(z - x_0 + g(0))^2\right\} \leq \frac{2C}{T}\exp\left\{-\frac{2c_0}{T}(z - x_0 + g(0))^2\right\},\quad t\in [T/2,T].
\end{equation*}
Using the dominated convergence theorem, we can pass to the limit in \eqref{psi_integ_parts} as $t\uparrow \T$ to obtain
\begin{equation*}
	\psi(T-)= \int_{-\infty}^0 \partial_yq(0,x_0;\T,y+g(\T))\,dy= 0,
\end{equation*}
where we used the boundary conditions $q(0,x_0;\T,g(\T))= q(0,x_0;\T,-\infty) =0.$ Theorem \ref{thm:Z} is proved.
\end{proof}

\begin{proof}[Proof of Theorem \ref{thm:factorisation}]
Since $\tau$ is a stopping time and $X_{\tau} = g(\tau),$ using Proposition~\ref{doob_transform} we obtain, for $0<t <t+\delta < T,$
\begin{equation*}
	\frac{1}{\delta}\mathbf{Q}(\tau \in (t,t+\delta)) = \frac{1}{\delta}\Ex  \frac{v'(\tau, X_{\tau})}{v'(0,x_0)} \bm{1}_{\tau \in (t,t+\delta )} = \frac{1}{\delta}\int_t^{t+\delta}\frac{v'(s,g(s))}{v'(0,x_0)}f_{\tau}(s)\,ds,
\end{equation*}
where the existence of $f_{\tau}$ was established in Lemma~\ref{lemma:fpt_cont}. Both $v'(s,g(s))$ and $f_{\tau}(s)$ are continuous at $t$ (by Lemmata \ref{lemma:boundary_grad} and \ref{lemma:fpt_cont}, respectively), and hence we get, after letting $\delta \to 0,$ that there exists the density~$f_{\tau}^{\mathbf{Q}}(t) $ of~$\tau$ under~$\mathbf{Q}$ that is equal to the right-hand side of~\eqref{eqn:doob_h}. Theorem~\ref{thm:factorisation} is proved.
\end{proof}

\begin{proof}[Proof of Theorem \ref{thm:dF}]
\textbf{Step 1.} We will first prove the assertion of Theorem \ref{thm:dF} in the case when $ h \in \cmSpace$. For $\delta \in (0,1),$ one has
\begin{equation}\label{eqn:thm1_1}
	F(g + \delta h) = \Px(X \in S_{g + \delta h}) = \Px(X - \delta h \in S_g).
\end{equation}
Note that there is no loss of generality restricting $\delta$ to be positive since the sign of $\delta$ can be transferred to $h.$ Consider the process
\begin{equation}\label{defn:M^h}
	M_t^h := -\int_0^t \frac{\dot{h}(s)}{\sigmaX(s,X_s)}\,dW_s,\quad t\in [0,\T].
\end{equation}
Denoting by
\begin{equation}\label{defn:quad_var_M}
	[M^h]_t = \int_0^t \frac{\dot{h}(s)^2}{\sigmaX^2(s,X_s)}\,ds,\quad t\in[0,\T],
\end{equation}
the quadratic variation process of $M^h,$ the Dol\'eans exponential of $M^h$ is given by
\begin{equation}\label{defn:doleans_exp}
	\mathcal{E}(\delta M^h)_t = \exp\{ \delta M_t^h - \mbox{$\frac{1}{2}$} \delta^2 [M^h]_t \}, \quad t \in [0,\T].
\end{equation}
From assumption (C2), one has
\begin{equation}\label{quadratic_var_bound}
	[M^h]_{\T} \leq \|h\|_{\cmSpace}^2/\usigma^2,
\end{equation}
so that Novikov's criterion is satisfied:
\begin{equation*}
	\Ex \exp\{\mbox{$\frac{1}{2}$}[\delta M^h]_{\T} \}\leq \exp\{\mbox{$\frac{1}{2\usigma^2}$}\|h\|_{\cmSpace}^2\} < \infty.
\end{equation*}
Hence $\mathcal{E}(\delta M^h)$ is a uniformly integrable martingale (see e.g.\ Proposition 1.15 on p.~332 in~\cite{RevuzYor}) and so we can apply the Cameron--Martin--Girsanov theorem (see e.g.\ Proposition~1.13 on p.~331 in~\cite{RevuzYor}) to get from \eqref{eqn:thm1_1} that
\begin{equation*}\label{girsanov}
	F(g+\delta h) = \Ex [\mathcal{E}(\delta M^h)_{\T} ; X \in S_g].
\end{equation*}
Since $\mathcal{E}(\delta M^h)_t$ satisfies $d\mathcal{E}(\delta M^h)_t = \delta\mathcal{E}(\delta M^h)_tdM_t^h$ (see e.g.\ p.~149 in~\cite{RevuzYor}), one has
\begin{equation}\label{defn:Z_eps}
	\Lmg_{\delta} := \frac{\mathcal{E}(\delta M^h)_{\T} - 1}{\delta} = \int_0^{\T} \mathcal{E}(\delta M^h)_s \,dM_s^h,
\end{equation}
and hence
\begin{equation}\label{diff_quotient}
	\frac{F(g + \delta h) - F(g)}{\delta} = \Ex [\Lmg_{\delta} ; X\in S_g].
\end{equation}
We will now show that $\{\Lmg_{\delta}\}_{\delta \in (0, 1)}$ is a uniformly integrable family of random variables. Using It\^o's isometry, representations \eqref{defn:quad_var_M} and \eqref{defn:Z_eps}, and condition (C2), we obtain
\begin{equation}\label{L2_norm_Z}
	\Ex\, \Lmg_{\delta}^2 = \Ex \int_0^{\T} \left( \frac{\mathcal{E}(\delta M^h)_s \dot{h}(s)}{\sigmaX(s,X_s)} \right)^2\,ds \leq \frac{\|h\|_{\cmSpace}^2}{\usigma^2}\Ex \sup_{s\in[0,\T]}\mathcal{E}(\delta M^h)_s^2.
\end{equation}
Since $\mathcal{E}(\delta M^h)$ is a martingale, by Doob's $\mathcal{L}^2$-martingale maximal inequality (see e.g.\ pp.~54--55 in~\cite{RevuzYor}) one has
\begin{equation*}
	\Ex \sup_{s\in[0,\T]}\mathcal{E}(\delta M^h)_s^2 \leq 4 \Ex  \mathcal{E}(\delta M^h)_{\T}^2.
\end{equation*}
As $M^h$ is a continuous martingale, by the Dambis--Dubins--Schwarz theorem (see e.g.\ Theorem~1.6 on p.~181 in~\cite{RevuzYor}), there exists a Brownian motion~$B$ on $(\Omega, \mathcal{F},\Px)$ such that $M_t^h = B_{[M^h]_t},$ $t\in[0,\T].$ Hence, using \eqref{defn:doleans_exp} and \eqref{quadratic_var_bound},
\begin{align*}
	\Ex\, \mathcal{E}(\delta M^h)_{\T}^2 &\leq \Ex \exp\{2\delta M_{\T}^h\} = \Ex \exp\{2 \delta B_{[M^h]_{\T}}\}\\
	&\leq \Ex \exp\left\{ 2 \max_{t \leq \|h\|_{\cmSpace}^2 / \usigma^2}B_t \right\} < \infty.
\end{align*}
This, together with \eqref{L2_norm_Z}, means that $\sup_{\delta \in (0,1)}\Ex \,\Lmg_{\delta}^2 <\infty$ so that $\{\Lmg_{\delta}\}_{\delta \in (0,1)}$ is a uniformly integrable family of random variables. Therefore we can pass to the limit as $\delta \downarrow 0$ on both sides of \eqref{diff_quotient}, establishing the existence of the derivative
\begin{equation}\label{hermite_rep}
	\nabla_h F(g) = \Ex \left[\lim_{\delta \downarrow 0}\Lmg_{\delta}; X\in S_g\right] = \Ex [M_{\T}^h ; X\in S_g]
\end{equation}
since $\lim_{\delta \to 0}\Lmg_{\delta} =M_{\T}^h$ a.s.\ from \eqref{defn:doleans_exp}.

Noting that $\{X \in S_g\}  = \{\tau = \T\}$, using \eqref{hermite_rep}, the martingale representation for $\bm{1}_{\tau = \T}$ from Lemma~\ref{MG_rep}, and the fact that $M^h$ is a zero-mean martingale, we have
\begin{equation*}
	\nabla_h F(g) =\Ex\,M_{\T}^h\bm{1}_{\tau = \T} = \Ex\, M_{\T}^h\int_0^{\T} (\V' \, \sigmaX)(t,X_t) \bm{1}_{\tau >t}\,dW_t.
\end{equation*}
Since $M^h$ is an It\^o integral of the form \eqref{defn:M^h}, by It\^o's isometry it follows that
\begin{equation}\label{eqn:MG_rep_DF}
	\nabla_h F(g) = -\Ex \int_0^{\T} \dot{h}(t)\V'(t,X_t)\bm{1}_{\tau>t}\,dt.
\end{equation}
We will now verify the conditions for applying Fubini's theorem to change the order of integration. By the Cauchy--Bunyakovsky--Schwartz inequality,
\begin{equation*}
	\Ex \int_0^{\T}\lvert \dot{h}(t)\V'(t,X_t)\bm{1}_{\tau >t}\rvert\,dt \leq \|h\|_{\cmSpace}\left(\Ex \int_0^{\T}\V'(t,X_t)^2\bm{1}_{\tau >t}\,dt\right)^{1/2} .
\end{equation*}
The last expression is finite since, applying It\^o's isometry to the martingale representation in Lemma \ref{MG_rep} and then using condition (C2), one has
\begin{align*}
	\frac{1}{4} &\geq \var(\bm{1}_{\tau = \T})=\Ex (\bm{1}_{\tau = \T} - v(0,x_0))^2 = \Ex \left( \int_0^{\T} (v'\,\sigma)(t,X_t)\bm{1}_{\tau >t}\,dW_t \right)^2\\
	&=\Ex \int_0^T (v'\,\sigma)^2(t,X_t)\bm{1}_{\tau > t}\,dt \geq \usigma^2 \Ex \int_0^T v'(t,X_t)^2\bm{1}_{\tau > t}\,dt.
\end{align*}
Hence the integrand the right-hand side of \eqref{eqn:MG_rep_DF} is absolutely integrable with respect to the product measure $\Px \times dt$ on $\Omega \times [0,T],$ so that we can change the order of integration in \eqref{eqn:MG_rep_DF} to obtain
\begin{equation}\label{eqn:df_riesz}
	\nabla_hF(g) = \int_0^T \dot{h}(t)\psi(t)\,dt,\quad h \in \cmSpace.
\end{equation}
Since $\psi$ is continuously differentiable by Theorem~\ref{thm:Z}, we can apply integration by parts to \eqref{eqn:df_riesz} together with the boundary conditions $h(0) =0,$ $\psi(T-) =0$ (Theorem~\ref{thm:Z}) to obtain \eqref{riesz_markov_rep} in the case when $h \in \cmSpace.$

\medskip

\textbf{Step 2.} Now we will turn to the case where $h \in \contSpace^2.$ Set $a:=h(0),$ $h_0(t) := h(t) - a,$ $t\in[0,T].$ For $\delta \in (0,1),$
\begin{equation}\label{eqn:decomp_shift}
	\frac{F(g + \delta h)-F(g)}{\delta}= \frac{F(g + \delta h) - F(g + \delta a)}{\delta}+ \frac{F(g + \delta a) - F(g)}{\delta}.
\end{equation}
Observe that the first term on the right-hand side has a limit as $\delta\downarrow 0$. Indeed, let $\phi(\theta) := F(g + \delta a + \theta h_0),$ $\theta \in [0,\delta].$ We showed in Step~1 that~$\phi$  is differentiable. Therefore, by the mean-value theorem, for any fixed $\delta \in (0,1),$ there exists a $\theta_{\delta} \in [0,\delta]$ such that
\begin{equation}\label{dF_h_a}
	\frac{F(g + \delta h)  - F(g + \delta a)}{\delta} = \frac{\phi(\delta) - \phi(0)}{\delta} = \phi'(\theta_{\delta}) = \nabla_{h_0}F(g + \delta a + \theta_{\delta}h_0),
\end{equation}
again using the result of  Step~1, this time with~$g$ replaced by $g+\delta a+\theta_{\delta}h_0 \in \contSpace^2$ (we assume without loss of generality that $x_0 < g(0) +\delta a$). Using \eqref{hermite_rep} and the Cauchy--Bunyakovsky--Schwartz inequality,
\begin{equation}\label{ineq:cauchy_ineq}
	\lvert \nabla_{h_0}F(g)-\nabla_{h_0}F(b)\rvert \leq \sqrt{\Ex (M_T^{h_0})^2\Px(X \in S_g\triangle S_b)}, \quad  b,g \in \contSpace,
\end{equation}
where $\triangle$ is the symmetric set difference. Letting $a_0 := \lvert a\rvert + \| h_0\|_{\infty},$ one has
\begin{equation*}
	S_g \triangle S_{g+\delta a + \theta_{\delta}h_0} \subseteq S_{g + \delta a_0} \backslash S_{g-\delta a_0} \downarrow \partial S_g \quad  \text{as } \delta \downarrow 0,
\end{equation*}
{so that
\begin{equation}\label{eqn:bndset_prob_0}
\Px(X \in S_g \triangle S_{g+\delta a + \theta_{\delta}h_0}) \to 0 \ \text{ due to }  \ \Px(X \in \partial S_g) =0,
\end{equation}
see e.g.\ p.~1406 in~\cite{Liang2023}.} Hence $\nabla_hF(g)$ is continuous in~$g$ in the uniform topology and
\begin{equation*}
	\lim_{\delta \downarrow 0}\nabla_{h_0}F(g + \delta a + \theta_{\delta}h_0) = \nabla_{h_0}F(g).
\end{equation*}
Now from here and \eqref{dF_h_a} we have
\begin{equation}\label{dF_h0}
	\lim_{\delta \downarrow 0}\frac{F(g + \delta h)  - F(g + \delta a)}{\delta} = \nabla_{h_0}F (g)= -\int_0^{\T} h_0(t)\dot{\psi}(t)\,dt
\end{equation}
using the result of Step 1 for $h_0 \in \cmSpace.$

Arguing in the same way as in the proof of Theorem~5 in~\cite{Borovkov2010}, it is not hard to show that the limit of the second term on the right-hand side of \eqref{eqn:decomp_shift} exists and is equal to
\begin{equation*}
	 \nabla_{a}F (g) = \lim_{\delta \downarrow 0}\frac{F(g+\delta a) - F(g)}{\delta} = -a \int_0^T \kappa(t) f_{\tau}(t)\,dt,
\end{equation*}
where $\kappa(t)$ is given by the right-hand side of \eqref{v_bd_deriv_meander} and is equal to $v'(t,g(t)).$ Now it follows from \eqref{eqn:doob_h} and \eqref{eqn:psi_dot_1} that
\begin{align*}
	 \nabla_{a}F (g) &= -a\int_0^T v'(t,g(t))f_{\tau}(t)\,dt \\
	&= -a\int_0^Tv'(0,x_0)f_{\tau}^{\mathbf{Q}}(t)\,dt = -a\int_0^T \dot{\psi}(t)\,dt.
\end{align*}
Combining this with~\eqref{eqn:decomp_shift} and~\eqref{dF_h0}, we complete the proof of Theorem~\ref{thm:dF}.
\end{proof}

\begin{remark}
The key difference between representation \eqref{riesz_markov_rep} and the one from \eqref{eqn:df_riesz} is that the latter one is only valid in the case when $h(0) =0 .$
\end{remark}

{\begin{proof}[Proof of Theorem~\ref{thm:dF_f}]
Using similarly to~\eqref{dF_h_a}, inequality~\eqref{ineq:cauchy_ineq}, and the fact that $\Ex (M_T^h)^2 \leq \| h\|_H^2/\usigma^2,$ one has,
\begin{align*}
	\left | F(g + \varepsilon h) - F(g) - \varepsilon \nabla_hF(g) \right |&= \left |\int_0^{\varepsilon} \frac{d}{d\theta}F(g + \theta h)\,d\theta - \varepsilon\nabla_h F(g)\right | \\
&=\left |\int_0^{\varepsilon} (\nabla_h F(g+ \theta h) - \nabla_h F(g))\,d \theta \right | \\
	&= \varepsilon\left |\int_0^1 (\nabla_h F(g+ \varepsilon u h) - \nabla_h F(g))\,d u \right |\\
		&\leq \varepsilon \sqrt{\Ex (M_T^h)^2} \sup_{u \in [0,1]}\sqrt{\Px(X \in S_{g + \varepsilon u h}\triangle S_g)}\\
		&\leq \varepsilon\frac{\| h\|_H}{\usigma} \sup_{u\in[0,1]}\sqrt{\Px(X \in S_{g + \varepsilon u h}\triangle S_g)},
\end{align*}
where $\sup_{u \in [0,1]}\Px(X \in S_{g + \varepsilon u h}\triangle S_g)\to 0$ as $\varepsilon \to 0$ by the same argument as in~\eqref{eqn:bndset_prob_0}. Taking suprema   in~$h$ subject to  $\|h\|_{\cmSpace} =1$ on both sides of the above inequality, we obtain the desired result.
\end{proof}
}

\begin{remark}\label{riesz_rep}
There is an alternative derivation of representation~\eqref{eqn:df_riesz} for $\nabla_h F,$ $h\in \cmSpace,$ and representation~\eqref{defn:psi} for $\psi(t)$ that is based on the Riesz representation formula for bounded linear functionals on Hilbert spaces, which is somewhat longer than our more compact argument using the martingale representation of $\bm{1}_{\tau =T}$. We will outline it below.

To begin with, it follows from \eqref{hermite_rep} and the linearity of expectations and the It\^o integral that $\nabla_h F(g)$ is a bounded linear functional on $\cmSpace.$ Hence, by the Riesz representation formula (see e.g.\ \S58 in~\cite{Kolmogorov1957b}), there exists a unique $k\in\cmSpace$ such that
\begin{equation}\label{riesz_rep_dF}
	\nabla_h F(g) = \langle h, k\rangle_{\cmSpace} = \int_0^{\T}\dot{h}(s)\dot{k}(s)\,ds,\quad h \in \cmSpace.
\end{equation}
Setting $h_t^* := s \wedge t$ for $s,t \in [0,\T],$ and
\begin{equation*}
	\rieszMG_t:= M_T^{h_t^*}= \int_0^t\frac{-dW_s}{\sigma(s,X_s)} , \quad t \in [0,T],
\end{equation*}
we have
\begin{equation}\label{defn:k}
	k(t) = \int_0^{\T} \dot{h}_t^*(s)\dot{k}(s)\,ds=  \Ex [M_T^{h_t^*};X\in S_g] = \Ex [ \rieszMG_t ; X\in S_g], \quad t\in[0,\T].
\end{equation}
By the Markov property of $X$, it follows that
\begin{equation}\label{eqn:psi_1}
	k(t) = \Ex \,\rieszMG_t\mathbf{1}_{\tau =T} = \Ex [\rieszMG_t \mathbf{1}_{\tau >t} \Ex [ \mathbf{1}_{\tau^{t,X_t}=T} \,|\,\mathcal{F}_t]]= \Ex \,\rieszMG_t v(t,X_t)\mathbf{1}_{\tau >t}.
\end{equation}
Using the It\^o calculus product rule and setting $R_t:=v(t,X_t),$ we have, since $U_0=0,$
\begin{equation}\label{eqn:Y_v}
	\rieszMG_{t\wedge \tau} R_{t\wedge \tau}= \int_0^{t\wedge \tau} \rieszMG_s\, dR_s+\int_0^{t\wedge \tau} R_s\,d\rieszMG_s + [U,R]_{t\wedge \tau},
\end{equation}
where $[\rieszMG,R]_{\cdot} $ denotes the quadratic covariation process of $U$ and $R$. Since $\rieszMG$ and $R$ are both uniformly integrable martingales, taking expected values on both sides of \eqref{eqn:Y_v} and using the optional stopping theorem, we obtain
\begin{equation}\label{stopped_quad_covar}
	\Ex [\rieszMG,R]_{t \wedge\tau }  = \Ex \,\rieszMG_{t\wedge\tau}R_{t\wedge\tau}=\Ex \,\rieszMG_t R_t\mathbf{1}_{\tau >t} +\Ex \,\rieszMG_{\tau}R_{\tau}\mathbf{1}_{\tau\leq t} = k(t)
\end{equation}
from~\eqref{eqn:psi_1}, since $R_{\tau} = v(\tau,X_{\tau}) = v(\tau,g(\tau)) =0$ on the event $\{\tau\leq t\},$ $t\in [0,\T).$ As $[R]_{t\wedge \tau} = \int_0^{t\wedge \tau}( v'\,\sigma)^2(s,X_s)\,ds$ and $[\rieszMG]_t = -\int_0^t ds/\sigma^2(s,X_s)$ (due to \eqref{v_MG} and \eqref{defn:quad_var_M}, respectively), we have
\begin{equation*}\label{quad_covar}
	[\rieszMG,R]_{t\wedge \tau}= -\int_0^{t\wedge \tau} v'(s,X_s)\,ds,\quad t\in [0,\T).
\end{equation*}
It follows from \eqref{stopped_quad_covar} that
\begin{equation}\label{eqn:k}
	k(t)  = - \Ex \int_0^{t\wedge\tau}v'(s,X_s)\,ds.
\end{equation}
Using the dominated convergence theorem, justified by the boundedness of $v''$ in $\text{cl}(D_{\T-\varepsilon}),$ $\varepsilon \in(0,\T),$ we conclude that the right-hand side of \eqref{eqn:k} is differentiable with respect to $t,$ and hence
\begin{equation*}
	\dot{k}(t) = -\Ex \,v'(t,X_t)\bm{1}_{\tau >t} = \psi(t).
\end{equation*}
Now \eqref{eqn:df_riesz} follows from \eqref{riesz_rep_dF}, which completes our derivation.

\end{remark}

\begin{remark}
In the case of the Brownian motion process $X=W,$ under some broad conditions on the mapping $J : \contSpace_0 \to \mathbb{R},$ a representation for the first variation in directions from~$\cmSpace$ of functionals of the form
\begin{equation*}
	g \mapsto \Ex [J(X) ; X \in S_g]
\end{equation*}
on the space of Borel measurable functions $g(t),$ $t\in [0,T],$ was derived in~\cite{Cameron1951}. For $J(\cdot) \equiv 1,$ that representation coincides with our \eqref{hermite_rep} in the special case when $X=W.$
It appears that paper~\cite{Cameron1951} might have been mostly overlooked in the literature devoted to the boundary-crossing problem for diffusion processes.
\end{remark}

\begin{remark}
After the authors had established \eqref{riesz_rep_dF} with the function $k$ given by \eqref{defn:k} using the Riesz representation, they became aware that, in the special case of $X = W,$ the same representation was obtained in~\cite{Gur2019}.
\end{remark}

{
\begin{proof}[Proof of Theorem~\ref{thm:dF_n}]
Since $h_n \in \cmSpace,$ $n\in\mathbb{N}$, the following representation 
follows from Theorem~\ref{thm:dF}:
\begin{equation}\label{eqn:dF_g_n}
	\nabla_{h_n}F(g)  = -\int_0^T h_n(t)\dot{\psi}(t)\,dt.
\end{equation}

First  we will show that,  for all indicator  functions   of the form $\phi (t) := \mathbf{1}_{(a,b]}(t) $ with $0\le a < b \le T,$ one has
\begin{equation}\label{eqn:FMCT}
\lim_{n\to\infty} \int_0^T h_n(t)\phi(t)\,dt = \frac{T^2}{12}\int_0^T \ddot{g}(t)\phi(t)\,dt.
\end{equation}

Introduce functions $q_{I}(t) :=  \frac{1}{2} (t-a)(b-t) \mathbf{1}_I(t),$ $I := (a,b)$ for $ a<b,$ and note that clearly
\begin{equation}\label{eqn:norm_gI}
	\|q_{I}\|_\infty = (b-a)^2/8. 
\end{equation}
As $\ddot{g}$ is uniformly continuous, setting $ I_{n,j} := (\frac{j-1}{n}T,\frac{j}{n}T )$, $j=1,\ldots, n,$ one has (see e.g.~Theorem 8.10 in~\cite{Kress1998})
\begin{equation*}
	h_n(t) = n^2\sum_{j=1}^n   q_{I_{n,j}}(t) (\ddot{g}(\mbox{$\frac{j}{n}T$}) + \theta_n(t)) ,
\end{equation*}
where $\| \theta_n\|_\infty \to 0 $    as $n\to\infty.$ From here and~\eqref{eqn:norm_gI} we get 
\begin{equation}\label{eqn:h_small}
\| h_n\|_\infty \le \mbox{$\frac{T^2}8$}\|\ddot g\|_\infty +o(1).
\end{equation}

Since clearly $n^2\int_0^T q_{I_{n,j}}(t)\,dt =n^2\int_{I_{n,j}}q_{I_{n,j}}(t)\,dt= T^3/(12n)$ and, due to~\eqref{eqn:norm_gI},
\[
\Big\| n^2 \theta_n \sum_{j=1}^n   q_{I_{n,j}}  \Big\|_\infty
 \le 
 \| \theta_n\|_\infty n^2 \max_{1\leq j \leq n} \|q_{I_{n,j}}\|_\infty
 = \mbox{$\frac{T^2}8$} \| \theta_n\|_\infty =o(1),
\]
it follows that
\begin{equation}\label{eqn:approx_integral}
	\int_0^T h_n(t)\phi (t)\,dt = \sum_{j= \lceil an/T \rceil}^{\lfloor b n/T\rfloor} \ddot{g}( \mbox{$\frac{j}{n}T$})\frac{T^3}{12n} +  o(1),\quad n\to\infty.
\end{equation}
As $\ddot{g}$ is continuous, the Riemann sum above converges to the    integral on the right-hand side of~\eqref{eqn:FMCT}. 

Next it follows by the linearity of the integral that~\eqref{eqn:FMCT} also holds for step functions~$\phi.$ Finally, for any $\phi\in L^1 :=L^1([0,T],dt)$ and  $\varepsilon >0,$ there exists a step function $\phi_0$ such that $\|\phi -\phi_0\|_{L^1}<\varepsilon,$ so that
\begin{align*}
\bigg|\int_0^T h_n (t) \phi (t) dt- \int_0^T h_n (t) \phi_0 (t) dt \bigg|
& \le 
\int_0^T |h_n (t) (\phi (t)-\phi_0 (t)) |dt
\\
&  \le \|h_n\|_\infty\cdot \|\phi -\phi_0\|_{L^1} <c\varepsilon
\end{align*}
for some $c<\infty$ in view of~\eqref{eqn:h_small}. Therefore we have
\begin{align*}
\limsup_{n\to\infty} 
 \int_0^T h_n (t) \phi (t) dt 
 & 
 \le 
 \lim_{n\to\infty} \int_0^T h_n (t) \phi_0 (t) dt + c\varepsilon
 \\
 & 
 =  \frac{T^2}{12} \int_0^T \ddot g (t) \phi_0 (t) dt  + c\varepsilon
 \le  \frac{T^2}{12}  \int_0^T \ddot g (t) \phi  (t) dt  + c_1 \varepsilon
\end{align*}
with $c_1 = \frac{T^2}{12}\|\ddot g\|_\infty +c.$ As a symmetric lower bound follows using the same argument and $\varepsilon>0$ is arbitrary small, we showed that~\eqref{eqn:FMCT} holds for any $\phi\in L^1.$   Since $\dot{\psi} \in L^1,$ Theorem~\ref{thm:dF_n} is proved.
\end{proof}
}

\section{Examples}\label{examples}

Only in the simplest cases can one obtain closed-form expressions for the boundary non-crossing probability, let alone its G\^ateaux derivative. However, it is still instructive to verify our formulae in two simple special cases that admit closed-form expressions for both the boundary crossing probability and its G\^ateaux derivative: the Brownian motion process and a linear boundary, and the Brownian bridge process and a linear boundary (Examples~\ref{exam_1} and \ref{exam_2} respectively).

 {Note that standard numerical solvers (e.g.\ \texttt{NDSolve} in {\sc Mathematica}) applied to \eqref{PDE} will provide approximate values of $v(s,x)$ for all $(s,x) \in \text{cl}(\D).$ Hence one can use finite differences to approximate $v'(t,g(t))$ and~$f_{\tau}(t),$ allowing for the approximation of $\dot{\psi}(t).$ Therefore one can obtain an approximation for $\nabla_h F (g)$ in the general case by numerically integrating \eqref{riesz_markov_rep}. We will demonstrate this in Example~\ref{exam_3} below.}

{Finally, we will numerically  confirm the claim of Theorem~\ref{thm:dF_n} in the special  case of the Brownian motion and the so-called Daniels boundary~$g$. Namely,  we will numerically show in Example~\ref{exam_4} that, in this particular  case, the sequence of derivatives $\nabla_{h_n}F(g_n),$ $n \in \mathbb{N},$ where $h_n =n^2(g_n -g)$ is the rescaled error arising from  approximation of~$g$ by the piecewise linear functions $g_n $, approaches $\nabla_{T^2\ddot{g}/12} F(g)$ as~$n$ grows. }

\begin{example}\label{exam_1}  \textit{Brownian motion and a linear boundary.}
For $t\in[0,\T],$ let $g(t) := a_1 + b_1t$ with $a_1>0$, $b_1\in \mathbb{R},$ and $h(t) := a_2 + b_2t$ with $a_2,b_2\in \mathbb{R}$. It is well-known (see e.g.\ 1.1.4 on p.~250 in~\cite{Borodin2002}) that, in the case of the Brownian motion process ($\muX \equiv 0,$ $\sigmaX \equiv 1$), for $(s,x) \in \D,$ the function $v$ defined in~\eqref{defn:v} equals
\begin{equation}\label{linear_bcp}
	v(s,x)=\Phi\left(\frac{a_1 + b_1 \T - x}{\sqrt{\T-s}} \right) -e^{-2b_1(a_1 + b_1s-x)}\Phi\left(\frac{x- a_1 + b_1\T-2b_1s}{\sqrt{\T-s}}\right),
\end{equation}
where $\Phi(z) := \int_{-\infty}^{z}\varphi(y)\,dy,$ $\varphi(y) := e^{-y^2/2}/\sqrt{2\pi}.$ Due to the space-homogeneity and time-scaling property of the Brownian motion, it suffices to analyse the case where $\T=1$ and $x_0=0.$ From \eqref{linear_bcp},
\begin{equation}\label{bd_deriv_linear}
	v'(t,g(t)) = -\mbox{$\sqrt{\frac{2}{\pi(1-t)}}$}e^{-\frac{1}{2}b_1^2 (1-t) } - 2b_1\Phi(b_1\sqrt{1-t}), \quad t\in(0,1),
\end{equation}
and so
\begin{align}\label{dF_BM}
	F(g) &= v(0,0) = \Phi(a_1+b_1) - e^{-2a_1b_1}\Phi(b_1-a_1),\nonumber\\
	\nabla_{h} F(g) &= 2a_2\varphi(a_1+b_1) +  2(a_1b_2 + b_1a_2)e^{-2a_1b_1}\Phi(b_1-a_1) .
\end{align}
Using the well-known formula for the density $f_{\tau}$ of the first hitting time of $g(t) = a_1+b_1t$:
\begin{equation}\label{bachelier_levy}
	f_{\tau}(t) = \frac{a_1}{t}\varphi_{t}(a_1+b_1t),\quad t >0,\quad \varphi_t(y) := \varphi(y/\sqrt{t})/\sqrt{t}, \quad t >0, \,y \in \mathbb{R},
\end{equation}
(see e.g.~2.0.2 on p.~198 in~\cite{Borodin2002}) one can verify numerically for concrete values of $a_1,b_1$ and $a_2,b_2$ that our integral representation \eqref{riesz_markov_rep} from Theorem \ref{thm:dF} yields the same results as \eqref{dF_BM}.

Now we turn to Lemma \ref{prop:VL_BD}. For our linear boundary $g$ one has (see \eqref{defn:eta})
\begin{equation*}
	\muZ(t,z) = -\dot{g}(t) = -b_1, \quad \Eta(t,z) = -b_1z,
\end{equation*}
so that, from \eqref{defn:G},
\begin{equation*}
	G_t(w) = \exp\{-b_1 w_{1-t} - \mbox{$\frac{1}{2}$}b_1^2 (1-t)\}, \quad w \in \contSpace ([0,1-t]).
\end{equation*}
Hence it follows that
\begin{equation*}
	 \Ex \,G_t(-W^{\oplus,1-t})= e^{-\frac{1}{2}b_1^2(1-t)}\Ex e^{b_1 W_{1-t}^{\oplus ,1-t}}, \quad t\in(0,1).
\end{equation*}
Recalling that (see e.g.\ (54) in~\cite{Borovkov2010})
\begin{equation*}
	\Ex e^{\lambda W_{1}^{\oplus,1}}= 1 + \sqrt{2\pi }\lambda e^{\lambda^2 /2}\Phi(\lambda ), \quad \lambda \in \mathbb{R},
\end{equation*}
and hence $\Ex e^{b_1W_{1-t}^{\oplus,1-t}} =  \Ex e^{b_1\sqrt{1-t}W_1^{\oplus,1}}$ by the scaling property of the Brownian meander: $\{W_{s}^{\oplus ,u}: s\in [0,u]\} \stackrel{d}{=} \{\sqrt{u}W_{s}^{\oplus,1}, s\in [0,1]\}$, we obtain, in our special case, that the right-hand side of \eqref{v_bd_deriv_meander} does coincide with \eqref{bd_deriv_linear}.

In the special case when $a_2=0,$ representation \eqref{riesz_rep_dF} for $\nabla_{h} F(g)$ becomes
\begin{equation*}
	\nabla_{h} F(g) = \langle h,k\rangle_{\cmSpace}= b_2\int_0^1 \dot{k}(t)\,dt = b_2k(1)
\end{equation*}
since $k(0)=0$ for $k\in \cmSpace$. Using the well-known fact that, for our $g(t) = a_1+b_1t$ and $x<a_1+b_1,$ one has (see e.g.\ p.~67 in~\cite{Borodin2002})
\begin{equation*}
	\Px(W \in S_g\,|\, W_1 = x) = 1 - e^{-2a_1(a_1 + b_1 -x)},
\end{equation*}
we obtain from \eqref{defn:k} with $U_t = -W_t$ that
\begin{align*}
	k(1)& =-\Ex [W_1; W\in S_g] = -\int_{-\infty}^{g(1)}x \Px(W\in S_g\,|\,W_1=x)\Px(W_1 \in dx)\\
	&= -\int_{-\infty}^{a_1+b_1}x(1-e^{-2a_1(a_1+b_1-x)})\varphi(x)\,dx =2 a_1 e^{-2a_1b_1}\Phi(b_1-a_1),
\end{align*}
which coincides with \eqref{dF_BM} when $a_2 =0.$

\begin{figure}[htbp]
	\centering
	\includegraphics[scale = 0.85]{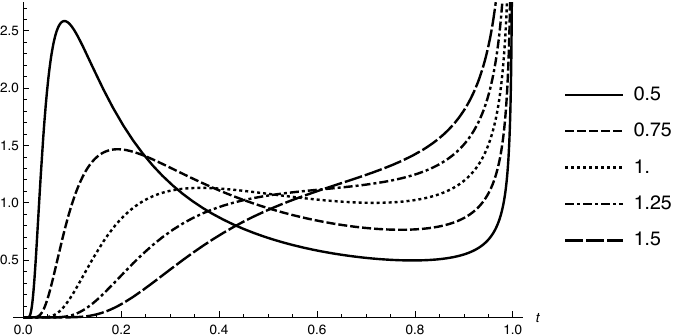}
 	\caption{The graphs of $f_{\tau}^{\mathbf{Q}}(t)$ for $g(t) = a_1+\frac{1}{2}t$ for different values of~$a_1$ in the case of the standard Brownian motion. The high values of $f_{\tau}^{\mathbf{Q}}$ in vicinity of $t=0$ for small~$a_1$ mean that the contribution of $h(t)$ for those~$t$'s to $\nabla_hF(g)$ is more significant in this situation.}
	\label{fig:bm1}
\end{figure}

As we observed after \eqref{eqn:VL}, the effect of the value of the direction $h(t)$ at time $t\in(0,\T)$ on the derivative $\nabla_h F$ is proportional to the value of $f_{\tau}^{\mathbf{Q}}(t).$ Representation \eqref{eqn:doob_h}, \eqref{v_bd_deriv_meander} for this density shows that it ``explodes'' as $t \uparrow T.$ This can be seen in Fig.~\ref{fig:bm1} depicting the graphs of $f_{\tau}^{\mathbf{Q}}$ in the special case considered in this example, for $a_1 = 0.5;0.75;1;1.25;1.5,$ and $b_1 = 1/2.$ Therefore, unless the boundary $g(t)$ is pretty close to~$x_0$ for small $t\geq0,$ $\nabla_hF$ is most sensitive to the values of $h(t)$ for $t$
close to $T.$ However, when $g(0)-x_0$ is small, the density $f_{\tau}^{\mathbf{Q}}(t)$ becomes large close to the time point $t=0$ as well (see Fig.~\ref{fig:bm1}), thus making the effect of the values of $h(t)$ for such $t$ on $\nabla_hF$ even more significant than that of $h(t)$ as $t\uparrow T$. The above considerations, which, from the probabilistic viewpoint make perfect sense, are to be taken into account when approximating general boundaries~$g$ with simpler ones to make computation of boundary crossing probabilities more feasible, cf.\ our discussion in Section~1.

We conclude this example with a comment on what our measure $\mathbf{Q}$ is in the special case when $b_1 =0$ (so that $g(t) \equiv a_1$). Recalling the definition of $\gamma$ in \eqref{defn:h}, we have in this case
\begin{align*}
	\gamma(t,y) = \frac{v''}{v'}(t,y)= \frac{a_1-y}{1-t}, \quad (t,y) \in D_1,
\end{align*}
which we recognise as the drift of a Brownian bridge pinned at the points $(0,0)$ and $(1,a_1).$ Denoting by $\{W_t^{\circ}:t\in[0,1]\}$ the standard Brownian bridge process on $[0,1],$ it follows from Proposition~\ref{doob_transform} that, under   measure $\mathbf{Q}$, our $X= W$ has the same dynamics on the time interval $[0,\tauX]$ as $W_t^{\circ} + a_1t$ prior to hitting $g(t) \equiv a_1.$
\end{example}

\begin{example}\label{exam_2} \textit{Brownian bridge and a linear boundary}.
For a fixed $y\in\mathbb{R},$ the Brownian bridge $\{X_t = W_t^{\circ} + yt : t\in[0,1]\}$ satisfies the following stochastic equation (see e.g.\ p.\,154 in~\cite{RevuzYor}):
\begin{equation*}
	X_t = \int_0^t\frac{y-X_s}{1-s}\,ds + W_t, \quad t \in [0,1).
\end{equation*}
Clearly, condition (C3) is not satisfied in what concerns the drift coefficient of our process. However, we will see below  that the integral representation for $\nabla_hF$ in Theorem~\ref{thm:dF} still holds, indicating that the conditions in (C3) may be relaxed.

The probability that our Brownian bridge~$X$ does not cross the linear boundary $g(t) = a_1+ b_1t$ with $a_1>0,$ $b_1\in\mathbb{R}$ and $a_1+b_1 >y,$ is well-known (see e.g.\ p.~67 in~\cite{Borodin2002}): for $ (s,x) \in D_1,$
\begin{equation}\label{brownian_bridge_bcp}
	v(s,x) =  1- \exp\{\mbox{$\frac{-2}{1-s}$}(a_1+b_1s - x)(a_1+b_1-y)\}.
\end{equation}
Let $h(t) := a_2 + b_2t$ with $a_2,b_2\in \mathbb{R}.$ From \eqref{brownian_bridge_bcp} we explicitly compute
\begin{equation}\label{eqn:dF_bb}
	\nabla_h F(g) = 2(2 a_1a_2 - a_2y + (b_1 a_2  + a_1 b_2) )e^{-2a_1(a_1+b_1-y)}.
\end{equation}
To compare this with our representation for $\nabla_h F$ from Theorem~\ref{thm:dF}, we first obtain from \eqref{brownian_bridge_bcp} that
\begin{equation*}\label{eqn:v'}
	 v'(t,g(t)) = -\frac{2(a_1+b_1-y)}{1-t},\quad t \in (0,1).
\end{equation*}
To compute $f_{\tau},$ recall that, for a general boundary $g_0,$ the first hitting time density of the Brownian bridge $X$ can be obtained in terms of the first hitting time density $\overline{f}_{\tau}(t)$ of that boundary by the Brownian motion: by Corollary~4 in~\cite{Borovkov2010},
\begin{equation*}
	f_{\tau}(t) = \frac{\varphi_{1-t}(y-g_0(t))}{\varphi_1(y)}\overline{f}_{\tau}(t), \quad t \in(0,1).
\end{equation*}
Now we get from \eqref{eqn:psi_dot_2} that, for the linear boundary $g(t) = a_1+b_1t,$ one has
\begin{equation*}\label{eqn:dpsi_BB}
	\dot{\psi}(t) = -\frac{2(a_1+b_1 - y)}{1-t}\frac{\varphi_{1-t}(y-a_1-b_1t )}{\varphi_1(y)}\overline{f}_{\tau}(t),\quad t\in (0,1),
\end{equation*}
where~$\overline{f}_{\tau}$ is given by the right-hand side of \eqref{bachelier_levy}.
Fig.~\ref{fig:bb1} illustrates the above formula, showing how the shape of the factor  $-\dot{\psi}(t)$ in representation~\eqref{riesz_markov_rep} for the derivative depends on the slope of the linear boundary.

One can verify numerically for concrete values of $a_1,b_1$ and $a_2,b_2$ that \eqref{riesz_markov_rep} with the above $\dot{\psi}$ yields the same results as \eqref{eqn:dF_bb}.

\begin{figure}[t]
	\centering
	\includegraphics[scale = 0.85]{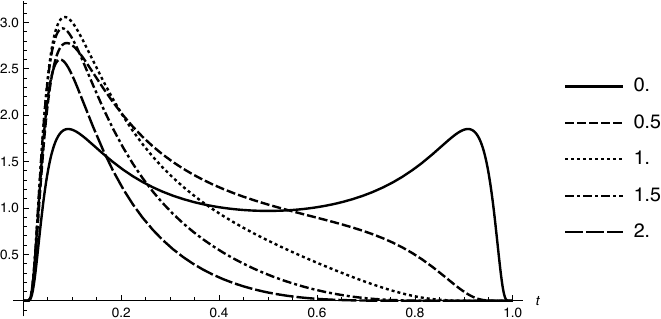}
 	\caption{The graph of $-\dot{\psi}(t)$ for $g(t) = \frac{1}{2}+b_1t$ for different values of $b_1,$ in the case of the Brownian bridge pinned at the points $(0,0)$ and $(1,0).$}
	\label{fig:bb1}
\end{figure}
\end{example}

\begin{example}\label{exam_3}\textit{Hyperbolic process and a hyperbolic boundary.}
{Let
\begin{equation*}
	g(t) = \frac{1}{2} + \frac{1}{4}\sin{3\pi t}, \quad t\geq0.
\end{equation*}
Suppose that $X$ is a diffusion process with drift coefficient
\begin{equation}\label{defn:drift_hyper}
	\mu(x) = \kappa \frac{1 - \lambda e^{-2\kappa x}}{1 + \lambda e^{-2\kappa x}}, \quad x,\kappa \in \mathbb{R}, \quad \lambda >0,
\end{equation}
and a unit diffusion coefficient. For this   process, closed-form expressions for the boundary crossing probabilities for a special class of boundaries  were obtained in~\cite{Crescenzo1997}. In the numerical example considered here, we set $\kappa = -1/2$ and $\lambda =1.$ By numerically solving the relevant boundary-value problems for the backward Kolmogorov equation~\eqref{PDE} and related forward equation in {\sc Mathematica}, we obtained approximations for $v'(t,g(t))$ and $f_{\tau}(t),$ respectively, for $t \in (0,1).$ Taking $h(t) = -\sin{3\pi t}$ and using numerical integration to evaluate the integral in~\eqref{riesz_markov_rep}, we obtained $\nabla_h F(g) \approx 0.155,$ which agrees with the finite difference approximation $(F(g + \delta h) - F(g))/\delta$ for $\delta = 0.001,$ obtained using our numerical solution to~\eqref{PDE}.}

{
The integral representation~\eqref{riesz_markov_rep} is useful for computational purposes since the density $-\dot{\psi}$ only needs to be calculated once, and then one can obtain $\nabla_h F(g)$ for arbitrary directions $h$ by evaluating a one-dimensional integral. Fig.~\ref{fig:ex3} illustrates the shapes of $-\dot{\psi}(t),$ $f_{\tau}(t)$ and $-v'(t,g(t))$ in the above setting.
\begin{figure}[t]
	\centering
	\includegraphics[scale = 0.85]{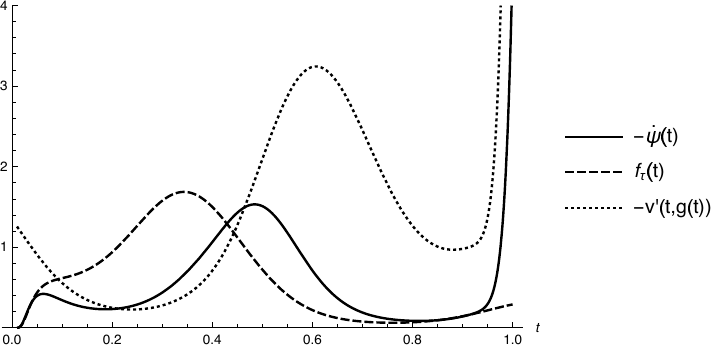}
 	\caption{The graphs of $-\dot{\psi}(t),$ $f_{\tau}(t)$ and $-v'(t,g(t))$ for $g(t) = \frac{1}{2}+\frac{1}{4}\sin{3\pi t},$ in the case of the diffusion process $X$ with drift coefficient given by~\eqref{defn:drift_hyper} and a unit diffusion coefficient.}
	\label{fig:ex3}
\end{figure}
}
\end{example}

{
\begin{example}\textit{Piecewise linear approximations of Daniels' boundary.}\label{exam_4}
The so-called Daniels boundary is defined as follows:
\begin{equation}\label{defn:daniels}
	g(t) := \frac{1}{2} - t \log\left( \frac{1}{4}\left(1 + \sqrt{1+8e^{-1/t}} \right) \right), \quad t >0,
\end{equation}
where $g(0+) = 1/2.$ A closed-form expression for $F(g)$ in the case of the Brownian motion was derived by Daniels~\cite{Daniels1969} using the method of images. We will consider the case $T=1.$ As before, for $n\in\mathbb{N},$ $g_n$ denotes the piecewise approximation of~$g$ with nodes at $(k/n,g(k/n)),$ $k=0,1\ldots,n.$   Using  {\sc Mathematica}, we numerically verified the following identity:
\begin{equation*}
	\lim_{n\to\infty} \nabla_{n^2(g_n- g)}F(g) = -\int_0^1 \frac{\ddot{g}(t)}{12} \dot{\psi}(t)\,dt \approx -0.0398.
\end{equation*}
Fig.~\ref{fig:ex4} illustrates the graphs of $\ddot{g}/12$ and $n^2(g_n -g)$ for $n=5.$
\begin{figure}[t]
	\centering
	\includegraphics[scale = 0.85]{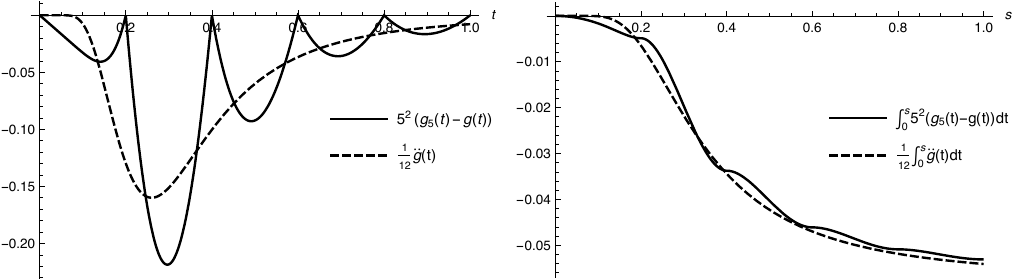}
 	\caption{{The left pane shows the graphs of $h_n(t) := n^2 (g_n(t) - g(t))$ for $n =5$ and $\ddot{g}(t)/12$ in the case of Daniels' boundary $g$ given by~\eqref{defn:daniels}. The right one shows how close the integrals of these functions over intervals $[0,s]$ for variable $s \in [0,1]$ are (cf.~\eqref{eqn:approx_integral}).}}
	\label{fig:ex4}
\end{figure}
\end{example}
}

\textbf{Acknowledgements.} The authors are grateful to N.\ V.\ Krylov for his advice on the differentiability properties of $v(s,x),$ which led to Lemma \ref{BKE}, and to X.~Geng and B.~Goldys for discussions on the problem.

\end{document}